\documentclass[10pt]{amsart}
\usepackage{amssymb,amsmath}
\usepackage{epsfig}
\usepackage{hyperref}
\usepackage{appendix} 
\newtheorem{theorem}{Theorem}[section]
\newtheorem{remm}{Remark}[section]
\newtheorem{lemma}{Lemma}[section]
\newtheorem{remark}[remm]{Remark}

\newtheorem{proposition}[theorem]{Proposition}

\numberwithin{equation}{section}
\newcommand{\aspro}{\vbox{\hrule width 6 pt
           \hbox{\vrule height 6 pt \hskip 5.2 pt \vrule height 6 pt}
           \hrule width 6 pt}}
\newcommand{\mybox}{\aspro}
\newcommand{\rset}{{\mathbb R}}
\newcommand{\nset}{{\mathbb N}}

\newcommand{\bfdot}{\bf\dot}
\newcommand{\ken}{\ \ }

\newcommand{\ssy}{\scriptscriptstyle}

\newcommand{\half}{\frac{1}{2}}
\begin{document}
\title[]
{Finite Element Approximations\\
for a linear Cahn-Hilliard-Cook equation\\
driven by the space derivative of a space-time white noise}
%
%
%
%
%
\author[]
{Georgios T. Kossioris$^{\ddag}$ and Georgios E.
Zouraris$^{\ddag}$}
\thanks{%
$^{\ddag}$Department of Mathematics, University of Crete, GR--714
09 Heraklion, Crete, Greece.}
%
%
%
%
%
%
%
\maketitle
%
%
%
%
%
\begin{abstract}
We consider an initial- and Dirichlet boundary- value problem for
a linear Cahn-Hilliard-Cook equation, in one space dimension,
forced by the space derivative of a space-time white noise.
First, we propose an approximate regularized stochastic parabolic
problem discretizing the noise using linear splines. Then
fully-discrete approximations to the solution of the
regularized problem are constructed using, for the discretization
in space, a Galerkin finite element method based on
$H^2-$piecewise polynomials, and, for time-stepping, the Backward
Euler method.
Finally, we derive strong a priori estimates for the modeling error and
for the numerical approximation error to the solution of the regularized problem.
\end{abstract}
%
%
%
\section{Introduction}\label{SECT1}
%
%
%
%
%
Let $T>0$, $D=(0,1)$ and $(\Omega,{\mathcal F},P)$ be a complete
probability space. Then we consider the following model initial-
and Dirichlet boundary- value problem for a linear
Cahn-Hilliard-Cook equation: find a stochastic function
$u:[0,T]\times{\overline D}\to\rset$ such that
\begin{equation}\label{PARAP}
\begin{gathered}
\partial_tu +\partial_x^4u+\mu\,\partial_x^2u
=\partial_x{\dot W}(t,x)
\quad\forall\,(t,x)\in (0,T]\times D,\\
\partial_x^{2 m}u(t,\cdot)\big|_{\ssy\partial D}=0
\quad\forall\,t\in(0,T],
\ken m=0,1,\\
u(0,x)=0\quad\forall\,x\in D,\\
\end{gathered}
\end{equation}
a.s. in $\Omega$, where ${\dot W}$ denotes a space-time white
noise on $[0,T]\times D$ (see, e.g., \cite{Walsh86},
\cite{KXiong}) and $\mu$ is a real constant for which there
exists $\kappa\in{\mathbb N}$ such that
\begin{equation}\label{mu_condition}
(\kappa-1)^2\,\pi^2\leq\mu<\kappa^2\,\pi^2,
\end{equation}
where ${\mathbb N}$ is the set of all positive integers.  
The above stochastic partial differential equation  combines two
independent characteristics. On the one hand it corresponds to the
linearization of the Cahn-Hilliard-Cook equation  around a
homogeneous initial state, in the spinodal region, that governs
the dynamics of spinodal decomposition in metal alloys; see e.g.
\cite{BlomMPW}, and references therein. On the other hand the
forcing noise  is a derivative of a space-time white noise that
physically arises in generalized Cahn-Hilliard equations, which
are equations of conservative type describing the evolution of an
order parameter in phase transitions (see \cite{Hohen}; cf.
\cite{KV03}, \cite{AKS}, \cite{Rogers1988}).
\par
The mild solution of the problem above (cf. \cite{DZ2007}) is
given by the formula
\begin{equation}\label{MildSol}
u(t,x)=\int_0^t\!\int_{\ssy D}\Psi(t-s;x,y)\,dW(s,y),
\end{equation}
where
\begin{equation}\label{GreenKernel}
\Psi(t;x,y)=-\sum_{k=1}^{\infty}e^{-\lambda_k^2\,(\lambda_k^2-\mu)
t} \,\varepsilon_k(x)\,\varepsilon_k'(y) \quad
\forall\,(t,x,y)\in(0,T]\times{\overline D}\times{\overline D},
\end{equation}
with $\lambda_k:=k\,\pi$ for $k\in\nset$, and
$\varepsilon_k(z):=\sqrt{2}\,\sin(\lambda_k\,z)$ for
$z\in{\overline D}$ and $k\in\nset$.
Observe that $\Psi(t;x,y)=-\partial_yG(t;x,y)$, where
$G(t;x,y)=\sum_{k=1}^{\infty}e^{-\lambda_k^2\,(\lambda_k^2-\mu) t}
\,\varepsilon_k(x)\,\varepsilon_k(y)$ for all
$(t,x,y)\in(0,T]\times{\overline D}\times{\overline D}$,
is the space-time Green kernel of the corresponding deterministic
parabolic problem: find a deterministic function
$w:[0,T]\times{\overline D}\to\rset$ such that
\begin{equation}\label{Det_Parab}
\begin{gathered}
\partial_tw +\partial_x^4w+\mu\,\partial_x^2w= 0
\quad\forall\,(t,x)\in (0,T]\times D,\\
\partial_x^{2m}w(t,\cdot)\big|_{\ssy\partial D}=0
\quad\forall\,t\in(0,T],\ken m=0,1,
\\
w(0,x)=w_0(x)\quad\forall\,x\in D.\\
\end{gathered}
\end{equation}
\par
The goal of the paper at hand is to propose and analyze a
methodology of constructing finite element approximations to $u$.
\subsection{The regularized problem}\label{The_Reg_Problem}
Our first step is to construct below an approximate to
\eqref{PARAP} regularized problem getting inspiration from the
work \cite{ANZ} for the stochastic heat equation with additive
space-time white noise (cf. \cite{KZ2010}, \cite{KZ2009}).
\par
Let $N_{\star}\in{\mathbb N}$, $\Delta{t}:=\frac{T}{N_{\star}}$,
$J_{\star}\in\nset$ and $\Delta{x}:=\frac{1}{J_{\star}}$.
Then, consider a partition of the interval $[0,T]$ with nodes
$(t_n)_{n=0}^{\ssy N_{\star}}$ and a partition of ${\overline D}$
with nodes $(x_j)_{j=0}^{\ssy J_{\star}}$, given by
$t_n:=n\,\Delta{t}$ for $n=0,\dots,N_{\star}$ and
$x_j:=j\,\Delta{x}$ for $j=0,\dots,J_{\star}$. Also, set
$T_n:=(t_{n-1},t_n)$ for $n=1,\dots,N_{\star}$, and
$D_j:=(x_{j-1},x_j)$ for $j=1,\dots,J_{\star}$.
\par
First, we let ${\mathcal S}_{\star}$ be the space of functions which
are continuous on ${\overline D}$ and piecewise linear over the
above specified partition of $D$, i.e.,
\begin{equation*}
{\mathcal S}_{\star}:=\left\{\,s\in C({\overline D};{\mathbb R}):
\quad s\big|_{\ssy D_j}\in{\mathbb P}^1(D_j)\ \ \text{\rm for}\ \
j=1,\dots,J_{\star}\,\right\}\subset H^1(D).
\end{equation*}
It is well-known that ${\rm dim}({\mathcal S}_{\star})=J_{\star}+1$
and that the functions $(\psi_i)_{i=1}^{\ssy J_{\star}+1}\subset
{\mathcal S}_{\star}$ defined by:
\begin{gather*}
\psi_1(x):=\tfrac{1}{\Delta{x}}\,(x_1-x)^{+},\quad \psi_{\ssy
J_{\star}+1}(x):=\tfrac{1}{\Delta{x}}
\,(x-x_{\ssy J_{\star}-1})^{+},\\
\psi_i(x):=\tfrac{1}{\Delta{x}}\,\left[\,(x-x_{i-2})\,{\mathcal
X}_{\ssy (x_{i-2},x_{i-1}]}+(x-x_{i})\,{\mathcal X}_{\ssy
(x_{i-1},x_{i}]}\,\right],\quad i=2,\dots,J_{\star},
\end{gather*}
consist the well-known {\it hat functions} basis of ${\mathcal S}_{\star}$, where, for any
$A\subset{\mathbb R}$,  by ${\mathcal X}_{\ssy A}$ we denote the
index function of $A$.
Next, consider the fourth-order linear stochastic parabolic
problem:
\begin{equation}\label{AC2}
\begin{gathered}
\partial_t{\widehat u} +\partial_x^4{\widehat u}
+\mu\,\partial_x^2{\widehat u}=\partial_x{\widehat W}
\quad\text{\rm in}\ken (0,T]\times D,\\
\partial_x^{2m}{\widehat u}(t,\cdot)\big|_{\ssy\partial D}=0
\quad\forall\,t\in(0,T],\ken m=0,1,\\
{\widehat u}(0,x)=0\quad\forall\,x\in D,\\
\end{gathered}
\end{equation}
\par\noindent
a.e. in $\Omega$, where:
\begin{equation*}\label{WNEQ1}
{\widehat W}(t,x) :=\tfrac{1}{\Delta{t}} \,\sum_{n=1}^{\ssy
N_{\star}}{\mathcal X}_{\ssy T_n}(t)\,\left[\,\sum_{\ell=1}^{\ssy
J_{\star}+1} \left(\,\sum_{m=1}^{\ssy J_{\star}+1}
G^{-1}_{\ell,m}\,R_{n,m}
\,\right)\,\psi_{\ell}(x)\,\right],\quad\forall \,(t,x)\in
[0,T]\times{\overline D},
\end{equation*}
$G$ is a real,  $(J_{\star}+1)\times(J_{\star}+1)$, symmetric and
positive definite matrix with
\begin{equation*}
G_{i,j}:=(\psi_j,\psi_i)_{\ssy 0,D},\quad
i,\,j=1,\dots,J_{\star}+1,
\end{equation*}
and
\begin{equation*}
R_{n,i}:=\int_{\ssy T_n}\!\int_{\ssy D}\psi_i(x)\;dW(t,x), \quad
i=1,\dots,J_{\star}+1,\quad n=1,\dots,N_{\star}.
\end{equation*}
%
%
The solution of the problem \eqref{AC2}, has the integral
representation (see, e.g., \cite{LMag})
\begin{equation}\label{HatUform}
\begin{split}
\widehat{u}(x,t)=&\,\int_0^t\!\!\!\int_{\ssy D}G(t-s;x,y)
\,\partial_y{\widehat W}(s,y)\,dsdy\\
=&\,\int_0^t\!\!\!\int_{\ssy D}\Psi(t-s;x,y) \,{\widehat
W}(s,y)\,dsdy, \quad\forall\,(t,x)\in[0,T]\times{\overline D}.\\
\end{split}
\end{equation}
%
%
%
%
\begin{remark}
A simple computation verifies that $G$ is a tridiagonal matrix
with $G_{1,1}=G_{\ssy
J_{\star}+1,J_{\star}+1}=\frac{\Delta{x}}{3}$,
$G_{i,i}=\frac{2\,\Delta{x}}{3}$ for $i=2,\dots,J_{\star}$, and
$G_{i,i+1}=\frac{\Delta{x}}{6}$ for $i=1,\dots,J_{\star}$. Since
$G$ is symmetric we have in addition that
$G_{i-1,i}=\frac{\Delta{x}}{6}$ for $i=2,\dots,J_{\star}+1$.
\end{remark}
%
%
\begin{remark}
Let ${\mathcal I}=\{(n,i):\,n=1,\dots,N_{\star},
\,i=1,\dots,J_{\star}+1\}$. Using the properties of the stochastic
integral (see, e.g., \cite{Walsh86}), we conclude that
$R_{n,i}\sim{\mathcal N}(0,\Delta{t}\,G_{i,i})$ for all
$(n,i)\in{\mathcal I}$. Also, we observe that ${\mathbb
E}[R_{n,i}\,R_{n',j}]=0$ for $(n,i)$, $(n',j)\in{\mathcal I}$ with
$n\not=n'$, and hence they are independent since they are
Gaussian. In addition, we have that ${\mathbb
E}[R_{n,i}\,R_{n,j}]=\Delta{t}\,G_{i,j}$ for $(n,i)$,
$(n,j)\in{\mathcal I}$.
Thus, for a given $n$ the random variables $(R_{n,i})_{i=1}^{\ssy
J_{\star}+1}$ are Gaussian and correlated, with correlation matrix
$\Delta{t}\,G$.
\end{remark}
%
%
%
%
%
%
\subsection{The numerical method}\label{The_Numerical_Method}
Our second step is to construct finite element approximations of the
solution ${\widehat u}$ to the regularized problem.
\par
Let $M\in{\mathbb N}$, $\Delta\tau:=\frac{T}{M}$,
$\tau_m:=m\,\Delta\tau$ for $m=0,\dots,M$, and
$\Delta_m:=(\tau_{m-1},\tau_m)$ for $m=1,\dots,M$.
Also, let $r\in\{2,3\}$,  and $M_h^r\subset H^2(D)\cap H_0^1(D)$ be
a finite element space consisting of functions which are piecewise
polynomials of degree at most $r$ over a partition of $D$ in
intervals with maximum mesh-length $h$.
Then, computable fully-discrete approximations of ${\widehat u}$
are constructed by using the Backward Euler finite element method,
which first sets
\begin{equation}\label{FullDE1}
{\widehat U}_h^0:=0
\end{equation}
and then, for $m=1,\dots,M$, finds ${\widehat U}_h^m\in M_h^r$
such that
\begin{equation}\label{FullDE2}
(\,{\widehat U}_h^m-{\widehat U}_h^{m-1},\chi\,)_{\ssy 0,D}
+\Delta\tau\,\left[\,(\,({\widehat U}_h^m)'',\chi''\,)_{\ssy 0,D}
+\mu\,(\,({\widehat U}_h^m)'',\chi\,)_{\ssy
0,D}\,\right]=\int_{\ssy\Delta_m}(\,\partial_x{\widehat
W},\chi\,)_{\ssy 0,D}\,d\tau
\end{equation}
for all $\chi\in M_h^r$, where $(\cdot,\cdot)_{\ssy 0,D}$ is the
usual $L^2(D)-$inner product.
%
%
%
%
\subsection{An overview of the paper and related references}
Our analysis first focus on the estimation of the modeling error,
i.e. the difference $u-{\widehat u}$, in terms of the
discretization parameters $\Delta{t}$ and $\Delta{x}$. Indeed,
working with the integral representation of $u$ and ${\widehat
u}$, we obtain (see Theorem~\ref{BIG_Qewrhma1})
\begin{equation}\label{Model_Error}
\max_{t\in[0,T]}\left\{\,\int_{\ssy\Omega}\left(\int_{\ssy
D}|u(t,x)-{\widehat
u}(t,x)|^2\;dx\right)\,dP\,\right\}^{\frac{1}{2}}\leq\,C_{\rm me}\,\big(
\,\epsilon^{-\frac{1}{2}}\,\Delta{x}^{\frac{1}{2}-\epsilon}
+{\Delta t}^{\frac{1}{8}}\,\big),
\quad\forall\,\epsilon\in(0,\tfrac{1}{2}],
\end{equation}
where $C_{\rm me}$ is a positive constant that is independent of $\Delta{x}$,
$\Delta{t}$ and $\epsilon$.
Next target in our analysis, is to provide the fully discrete
approximations of ${\widehat u}$ defined in
Section~\ref{The_Numerical_Method} with a convergence result,
which is achieved by proving the following strong error estimate
(see Theorem~\ref{FFQEWR})
\begin{equation}\label{FDEstim4}
\max_{0\leq{m}\leq\ssy M}\left\{\int_{\ssy\Omega}\left( \int_{\ssy
D}\big|{\widehat U}_h^m(x)-{\widehat
u}(\tau_m,x)\big|^2\;dx\right)dP\right\}^{\frac{1}{2}}\leq\,C_{\rm ne}\,
\left(\,\,\epsilon_1^{-\frac{1}{2}}\,\Delta\tau^{\frac{1}{8}-\epsilon_1}
+\epsilon_2^{-\frac{1}{2}}\,h^{\nu(r)-\epsilon_2}\,\right),
\end{equation}
for all $\epsilon_1\in(0,\frac{1}{8}]$ and
$\epsilon_2\in(0,\nu(r)]$ with $\nu(2)=\tfrac{1}{3}$ and
$\nu(3)=\tfrac{1}{2}$, where $C_{\rm ne}$ is a positive constant
independent of $\epsilon_1$, $\epsilon_2$, $\Delta\tau$, $h$, $\Delta{x}$
and $\Delta{t}$.
To get the error estimate \eqref{FDEstim4} we use as an auxilliary
tool the Backward-Euler time-discrete approximations of ${\widehat
u}$ which are defined in Section~\ref{SECTION3}. Thus, we can see
the numerical approximation error as a sum of two types of error:
the {\it time-discretization} error and the {\it
space-discretization} error. The {\it time-discretization} error
is the approximation error of the Backward Euler time-discrete
approximations which is estimated in
Theorem~\ref{TimeDiscreteErr1}, while the {\it
space-discretization} error is the error of approximating the
Backward Euler time-discrete approximations by the Backward Euler
finite element approximations, which is estimated in
Proposition~\ref{Tigrakis}.
\par
Let us expose some related bibliography. The work \cite{Printems}
contains a general convergence analysis for a class of
time-discrete approximations to the solution of stochastic
parabolic problems, the assumptions of which may cover problem
\eqref{PARAP}. However, the approach we adopt here is different
since first we introduce a space-time discretization of the noise and then
we analyze time-discrete approximations to the solution. We would like to 
note that we are not aware of another work providing a rigorous
convergence analysis for fully discrete finite element
approximations to a stochastic parabolic equation forced by the
space derivative of a space-time white noise.
We refer the reader to our previous work \cite{KZ2010},
\cite{KZ2009} and to \cite{LM2009} for the construction and the
convergence analysis of Backward Euler finite element
approximations of the solution to the problem \eqref{PARAP} when
$\mu=0$ and an additive space-time white noise $\dot W$  is forced
instead of $\partial_x{\dot W}$.
Finally, we refer the reader to \cite{GK}, \cite{ANZ},
\cite{KloedenShot}, \cite{BinLi}, \cite{YubinY05} and
\cite{Walsh05} for the analysis of the finite element method for
second order stochastic parabolic problems forced by an additive
space-time white noise.
%
%
\par
We close the section by an overview of the paper.
Section~\ref{SECTIONCHILD} introduces notation, and recalls or
proves several results often used in the paper.
Section~\ref{SECTION2} is dedicated to the estimation of the
modeling error.
Section~\ref{SECTION3} defines the Backward Euler time-discrete
approximations of ${\widehat u}$ and analyzes its convergence.
%
%
Section~\ref{SECTION44} contains the error analysis for the
Backward Euler fully-discrete approximations of ${\widehat u}$.
%
%
%
%
%
%
%
%
%
%
%
%
%
%
%
%
\section{Notation and Preliminaries}\label{SECTIONCHILD}
\subsection{Function spaces and operators}\label{Section2.2}
Let $I\subset{\mathbb R}$ be a bounded interval. We denote by
$L^2(I)$ the space of the Lebesgue measurable functions which are
square integrable on $I$ with respect to Lebesgue's measure $dx$,
provided with the standard norm $\|g\|_{\ssy 0,I}:=
\left(\int_{\ssy I}|g(x)|^2\,dx\right)^{\frac{1}{2}}$ for $g\in
L^2(I)$. The standard inner product in $L^2(I)$ that produces the
norm $\|\cdot\|_{\ssy 0,I}$ is written as $(\cdot,\cdot)_{\ssy
0,I}$, i.e., $(g_1,g_2)_{\ssy 0,I} :=\int_{\ssy
I}g_1(x)g_2(x)\,dx$ for $g_1$, $g_2\in L^2(I)$.
Let ${\mathbb N}_0$ be the set of the nonnegative integers.
For $s\in\nset_0$, $H^s(I)$ will be the Sobolev space of functions
having generalized derivatives up to order $s$ in the space
$L^2(I)$, and by $\|\cdot\|_{\ssy s,I}$ its usual norm, i.e.
$\|g\|_{\ssy s,I}:=\left(\sum_{\ell=0}^s
\|\partial^{\ell}g\|_{\ssy 0,I}^2\right)^{\frac{1}{2}}$ for $g\in
H^s(I)$. Also, by $H_0^1(I)$ we denote the subspace of $H^1(I)$
consisting of functions which vanish at the endpoints of $I$ in
the sense of trace. We note that in $H_0^1(I)$ the, well-known,
Poincar{\'e}-Friedrich inequality holds, i.e., there exists a nonegative
constant $C_{\ssy P\!F}$ such that
\begin{equation}\label{Poincare}
\|g\|_{\ssy 0,I}\leq\,C_{\ssy P\!F}\,\|\partial g\|_{\ssy 0,I}
\quad\forall\,g\in H^1_0(I).
\end{equation}
\par
The sequence of pairs
$\left(\,(\lambda_k^2,\varepsilon_k)\,\right)_{k=1}^{\infty}$ is a
solution to the eigenvalue/eigenfunction problem: find nonzero
$\varphi\in H^2(D)\cap H_0^1(D)$ and $\sigma\in\rset$ such that
$-\partial^2\varphi=\sigma\,\varphi$ in $D$.
Since $(\varepsilon_k)_{k=1}^{\infty}$ is a complete
$(\cdot,\cdot)_{\ssy 0, D}-$orthonormal system in $L^2(D)$, for
$s\in\rset$, a subspace ${\mathcal V}^s(D)$ of $L^2(D)$ is defined by
\begin{equation*}
{\mathcal V}^s(D):=\left\{v\in L^2(D):\quad\sum_{k=1}^{\infty}
\lambda_{k}^{2s} \,(v,\varepsilon_k)^2_{\ssy
0,D}<\infty\,\right\}
\end{equation*}
which is  provided with the norm
$\|v\|_{\ssy{\mathcal V}^s}:=\big(\,\sum_{k=1}^{\infty}
\lambda_{k}^{2s}\,(v,\varepsilon_k)^2_{\ssy
0,D}\,\big)^{\frac{1}{2}} \quad\forall\,v\in{\mathcal V}^s(D)$.
For $s\ge 0$, the pair $({\mathcal V}^s(D),\|\cdot\|_{\ssy{\mathcal V}^s})$ is a complete
subspace of $L^2(D)$ and we set $({\bfdot H}^s(D),\|\cdot\|_{\ssy{\bfdot H}^s})
:=({\mathcal V}^s(D),\|\cdot\|_{\ssy{\mathcal V}^s})$.
For $s<0$, we define $({\bfdot H}^s(D),\|\cdot\|_{\ssy{\bfdot H}^s})$  as the completion of
$({\mathcal V}^s(D),\|\cdot\|_{\ssy{\mathcal V}^s})$, or, equivalently, as the dual of
 $({\bfdot H}^{-s}(D),\|\cdot\|_{\ssy{\bfdot H}^{-s}})$.
%
%
%
%
%
%
%
Let $m\in\nset_0$. It is well-known (see \cite{Thomee}) that
\begin{equation}\label{dot_charact}
{\bfdot H}^m(D)=\big\{\,v\in H^m(D):
\quad\partial^{2i}v\left|_{\ssy\partial D}\right.=0
\quad\text{\rm if}\ken 0\leq{i}<\tfrac{m}{2}\,\big\}
\end{equation}
and there exist positive constants $C_{m,{\ssy A}}$ and $C_{m,{\ssy B}}$
such that
\begin{equation}\label{H_equiv}
C_{m,{\ssy A}}\,\|v\|_{\ssy m,D} \leq\|v\|_{\ssy{\bfdot H}^m}
\leq\,C_{m,{\ssy B}}\,\|v\|_{\ssy m,D},\quad \forall\,v\in{\bfdot
H}^m(D).
\end{equation}
Also, we define on $L^2(D)$ the negative norm $\|\cdot\|_{\ssy -m,
D}$ by
\begin{equation*}
\|v\|_{\ssy -m, D}:=\sup\Big\{ \tfrac{(v,\varphi)_{\ssy 0,D}}
{\|\varphi\|_{\ssy m,D}}:\quad \varphi\in{\bfdot H}^m(D)
\ken\text{\rm and}\ken\varphi\not=0\Big\}, \quad\forall\,v\in
L^2(D),
\end{equation*}
for which, using \eqref{H_equiv}, it is easy to conclude that
there exists a constant $C_{-m}>0$ such that
\begin{equation}\label{minus_equiv}
\|v\|_{\ssy -m,D}\leq\,C_{-m}\,\|v\|_{\ssy{\bfdot H}^{-m}},
\quad\forall\,v\in L^2(D).
\end{equation}
\par
Let ${\mathbb L}_2=(L^2(D),(\cdot,\cdot)_{\ssy 0,D})$ and
${\mathcal L}({\mathbb L}_2)$ be the space of linear, bounded
operators from ${\mathbb L}_2$ to ${\mathbb L}_2$. We say that, an
operator $\Gamma\in {\mathcal L}({\mathbb L}_2)$ is {\sl
Hilbert-Schmidt}, when $\|\Gamma\|_{\ssy\rm
HS}:=\left(\sum_{k=1}^{\infty} \|\Gamma(\varepsilon_k)\|^2_{\ssy
0,D}\right)^{\half}<+\infty$, where $\|\Gamma\|_{\ssy\rm HS}$ is
the so called Hilbert-Schmidt norm of $\Gamma$.
%
%
We note that the quantity $\|\Gamma\|_{\ssy\rm HS}$ does not change
when we replace $(\varepsilon_k)_{k=1}^{\infty}$ by another
complete orthonormal system of ${\mathbb L}_2$, as it is the
sequence $(\varphi_k)_{k=0}^{\infty}$ with $\varphi_0(z):=1$ and
$\varphi_k(x):=\sqrt{2}\,\cos(\lambda_k\,z)$ for $k\in{\mathbb N}$
and $z\in{\overline D}$.
It is well known (see, e.g., \cite{DunSch}) that an operator
$\Gamma\in{\mathcal L}({\mathbb L}_2)$ is Hilbert-Schmidt iff there
exists a measurable function $g:D\times D\rightarrow{\mathbb R}$
such that $(\Gamma(v))(\cdot)=\int_{\ssy D}g(\cdot,y)\,v(y)\,dy$ for
$v\in L^2(D)$, and then, it holds that
%
%
\begin{equation}\label{HSxar}
\|\Gamma\|_{\ssy\rm HS} =\left(\int_{\ssy D}\!\int_{\ssy
D}g^2(x,y)\,dxdy\right)^{\half}.
\end{equation}
Let ${\mathcal L}_{\ssy\rm HS}({\mathbb L}_2)$ be the set of
Hilbert Schmidt operators of ${\mathcal L}({\mathbb L}^2)$ and
${\Phi}:[0,T]\rightarrow {\mathcal L}_{\ssy\rm
HS}({\mathbb L}_2)$. Also, for a random variable $X$, let
${\mathbb E}[X]$ be its expected value, i.e., ${\mathbb
E}[X]:=\int_{\ssy\Omega}X\,dP$.
Then, the It{\^o} isometry property for stochastic integrals,
which we will use often in the paper, reads
\begin{equation}\label{Ito_Isom}
{\mathbb E}\left[\Big\|\int_0^{\ssy
T}{\Phi}\;dW\Big\|_{\ssy 0,D}^2\right] =\int_0^{\ssy
T}\|{\Phi}(t)\|_{\ssy\rm HS}^2\,dt.
\end{equation}
\par
Let ${\widehat\Pi}:L^2((0,T)\times D) \rightarrow L^2((0,T)\times
D)$ be a projection operator defined by
\begin{equation}\label{Defin_L2}
{\widehat \Pi}g(t,x):=\tfrac{1}{{\Delta t}}\,\sum_{i=1}^{\ssy
J_{\star}+1}\left(\sum_{\ell=1}^{\ssy J_{\star}+1}
G^{-1}_{i,\ell}\int_{\ssy T_n}\!\int_{\ssy
D}g(s,y)\,\psi_{\ell}(y)\;dsdy\right)\,\psi_i(x),
\quad\forall\,(t,x)\in T_n\times D,
\end{equation}
for $n=1,\dots,N_{\star}$ and for $g\in L^2((0,T)\times D)$, for
which holds that
\begin{equation}\label{L2L2_Bound}
\left(\int_0^{\ssy T}\!\!\!\int_{\ssy
D}({\widehat\Pi}g)^2\;dxdt\right)^{\frac{1}{2}}\leq\,\left(\int_0^{\ssy
T}\!\!\!\int_{\ssy
D}g^2\;dxdt\right)^{\frac{1}{2}},\quad\forall\,g\in L^2((0,T)\times
D).
\end{equation}
Now, in the lemma below, we relate the stochastic integral of the
projection ${\widehat\Pi}$ of a deterministic function to its
space-time $L^2-$inner product with the discrete space-time white
noise kernel ${\widehat W}$ defined in
Section~\ref{The_Reg_Problem} (cf. Lemma~2.1 in \cite{KZ2010}).
%
%
%
\begin{lemma}\label{Lhmma1}
For $g\in L^2((0,T)\times D)$, it holds that
\begin{equation}\label{WNEQ2}
\int_0^{\ssy T}\!\!\!\int_{\ssy D}\,{\widehat\Pi}g(t,x)\,dW(t,x)
=\int_0^{\ssy T}\!\!\!\int_{\ssy D}\,{\widehat
W}(s,y)\,g(s,y)\,dsdy.
\end{equation}
\end{lemma}
%
%
%
%
%
%
%
%
%
%
%
%
%
%
%
%
%
%
\begin{proof}
To obtain \eqref{WNEQ2} we work, using \eqref{Defin_L2} and the
properties of the stochastic integral, as follows:
\begin{equation*}
\begin{split}
\int_0^{\ssy T}\!\!\!\int_{\ssy D}{\widehat\Pi}g(t,x)\,dW(t,x)
=&\tfrac{1}{\Delta{t}}\, \sum_{n=1}^{\ssy N_{\star}}
\sum_{i=1}^{\ssy J_{\star}+1} \sum_{\ell=1}^{\ssy
J_{\star}+1}G^{-1}_{i,\ell}\left(
\int_{\ssy T_n\times D}g(s,y)\,\psi_{\ell}(y)\;dsdy\right)\, R_{n,i}\\
=&\tfrac{1}{\Delta{t}}\, \sum_{n=1}^{\ssy N_{\star}} \int_{\ssy
T_n\times D} g(s,y)\,\left( \,\sum_{i=1}^{\ssy J_{\star}+1}
\sum_{\ell=1}^{\ssy J_{\star}+1}G^{-1}_{i,\ell}\psi_{\ell}(y)
\,R_{n,i}\right)\;dsdy\\
=&\tfrac{1}{\Delta{t}}\, \sum_{n=1}^{\ssy N_{\star}}\,
\int_0^{\ssy T}\!\!\int_{\ssy D} {\mathcal X}_{\ssy
T_n}(s)\,g(s,y)\,\left( \,\sum_{i=1}^{\ssy J_{\star}+1}
\sum_{\ell=1}^{\ssy
J_{\star}+1}G^{-1}_{\ell,i}\,R_{n,i}\,\psi_{\ell}(y)
\,\right)\;dsdy\\
=&\int_0^{\ssy T}\!\!\!\int_{\ssy D}g(s,y) \,{\widehat
W}(s,y)\,dsdy.
\end{split}
\end{equation*}
\end{proof}
%
%
%
%
\par
We close this section by observing that: if $c_{\star}>0$, then
\begin{equation}\label{SR_BOUND}
\sum_{k=1}^{\infty}\lambda_k^{-(1+c_{\star}\epsilon)}
\leq\,\left(\tfrac{1+2c_{\star}}{c_{\star}\pi}\right)\,\tfrac{1}{\epsilon},
\quad\forall\,\epsilon\in(0,2],
\end{equation}
and if $({\mathcal H},(\cdot,\cdot)_{\ssy{\mathcal H}})$ is a
real inner product space, then 
\begin{equation}\label{innerproduct}
(g-v,g)_{\ssy{\mathcal H}}
\ge\,\tfrac{1}{2}\,\left[\,(g,g)_{\ssy{\mathcal H}}
-(v,v)_{\ssy{\mathcal H}}\,\right],\quad\forall\,g,v\in{\mathcal H}.
\end{equation}
%
%
%
\subsection{Linear elliptic and parabolic
operators}\label{SECTION31}
%
Let us define the elliptic differential operators $\Lambda_{\ssy
B},\,{\widetilde\Lambda}_{\ssy B}:{\bfdot H}^4(D)\rightarrow
L^2(D)$ by $\Lambda_{\ssy B}v:=\partial^4v+\mu\,\partial^2v$ and
${\widetilde\Lambda}_{\ssy B}v:=\Lambda_{\ssy B}v+\mu^2\,v$ for
$v\in{\bfdot H}^4(D)$, and consider the corresponding Dirichlet
fourth-order two-point boundary value problems: given $f\in
L^2(D)$ find $v_{\ssy B}$, ${\widetilde v}_{\ssy B}\in {\bfdot
H}^4(D)$ such that
\begin{equation}\label{ElOp2}
\Lambda_{\ssy B}v_{\ssy B}=f\quad\text{\rm in}\ken D
\end{equation}
and
\begin{equation}\label{ElOp2_mu}
{\widetilde\Lambda}_{\ssy B}{\widetilde v}_{\ssy
B}=f\quad\text{\rm in}\ken D.
\end{equation}
Assumption \eqref{mu_condition} yields that when $\kappa=1$ or
$\kappa\ge2$ and $\mu\not=\lambda_{\kappa-1}^2$, the operator
$\Lambda_{\ssy B}$ is invertible and
thus the problem \eqref{ElOp2} is well-posed. However, the problem
\eqref{ElOp2_mu} is always well-posed.
Letting $T_{\ssy B},\,{\widetilde T}_{\ssy
B}:L^2(D)\rightarrow{\bfdot H}^4(D)$ be the solution operator of
\eqref{ElOp2} and \eqref{ElOp2_mu}, respectively, i.e. $T_{\ssy
B}f:=\Lambda_{\ssy B}^{-1}f=v_{\ssy B}$ and ${\widetilde T}_{\ssy
B}f:={\widetilde\Lambda}_{\ssy B}^{-1}f={\widetilde v}_{\ssy B}$,
it is easy to verify that
\begin{equation}\label{Akribhs_Typos}
T_{\ssy B}f=\sum_{k=1}^{\infty}\tfrac{(\varepsilon_k,f)_{\ssy
0,D}}{\lambda_k^2(\lambda_k^2-\mu)}\,\,\varepsilon_k\quad\text{\rm
and}\quad{\widetilde T}_{\ssy
B}f=\sum_{k=1}^{\infty}\tfrac{(\varepsilon_k,f)_{\ssy
0,D}}{\lambda_k^2(\lambda_k^2-\mu)+\mu^2}
\,\,\,\varepsilon_k,\quad\forall\,f\in L^2(D),
\end{equation}
and
\begin{equation}\label{ElBihar2}
\|T_{\ssy B}f\|_{\ssy m,D}+\|{\widetilde T}_{\ssy B}f\|_{\ssy
m,D}\leq \,C_{\ssy R,m}\,\|f\|_{\ssy m-4, D}, \quad\forall\,f\in
H^{\max\{0,m-4\}}(D), \ken\forall\,m\in{\mathbb N}_0,
\end{equation}
where $C_{\ssy R,m}$ is a positive constant which is independent of $f$ but depends on 
the $D$ and $m$.
Observing that
\begin{equation*}\label{TB-prop1}
({\widetilde T}_{\ssy B}v_1,v_2)_{\ssy 0,D}
 =(v_1,{\widetilde T}_{\ssy B}v_2)_{\ssy 0,D},
\quad\forall\,v_1,v_2\in L^2(D),
\end{equation*}
and in view \eqref{Akribhs_Typos}, the
map ${\widetilde\gamma}_{\ssy B}:L^2(D)\times L^2(D)\rightarrow{\mathbb R}$
defined by
\begin{equation*}
{\widetilde\gamma}_{\ssy B}(v,w)=({\widetilde T}_{\ssy B}v,w)_{\ssy 0,D}\quad\forall\,
v, w\in L^2(D), 
\end{equation*}
is an inner product on $L^2(D)$.
%
%
%
\par
Let $({\mathcal S}(t)w_0)_{\ssy t\in[0,T]}$ be the standard
semigroup notation for the solution $w$ of \eqref{Det_Parab}.
Then, the following a priori bounds hold (see
Appendix~\ref{APP_1}):
for $\ell\in{\mathbb N}_0$, $\beta\ge0$ and $p\ge0$,
there exists a constant ${\mathcal C}_{\beta,\ell,\mu,\mu T}>0$ such that:
\begin{equation}\label{Reggo3}
\int_{t_a}^{t_b}(\tau-t_a)^{\beta}\,
\big\|\partial_t^{\ell}{\mathcal S}(\tau)w_0 \big\|_{\ssy {\bfdot
H}^p}^2\,d\tau \leq\,{\mathcal C}_{\beta,\ell,\mu,\mu T}\, \|w_0\|^2_{\ssy {\bfdot
H}^{p+4\ell-2\beta-2}}
\end{equation}
forall $w_0\in{\bfdot H}^{p+4\ell-2\beta-2}(D)$ and $t_a$, $t_b\in[0,T]$ with $t_b>t_a$.
%
%
%
%
\subsection{Discrete spaces and operators}
%
%
For $r\in\{2,3\}$,  let $M_h^r\subset
H_0^1(D)\cap H^2(D)$ be a finite element space
consisting of functions which are piecewise polynomials
of degree at most $r$ over a partition of $D$ in
intervals with maximum mesh-length $h$. It is well-known (cf.,
e.g., \cite{BrHilbert1970}) that the following approximation
property holds:
\begin{equation}\label{Sh_H2}
\inf_{\chi\in M_h^r} \|v-\chi\|_{\ssy 2,D} \leq\,C_{\ssy
F\!M,r}\,h^{s-1}\,\|v\|_{\ssy s+1,D}, \quad\,\forall\,v\in
H^{s+1}(D)\cap H_0^1(D),\quad\forall\,s\in\{2,r\},
\end{equation}
where $C_{{\ssy F\!M},r}$ is a positive constant
that depends on $r$ and is independent of $h$ and $v$.
Then,  we define the discrete elliptic operators $\Lambda_{\ssy B,h},
\,{\widetilde\Lambda}_{\ssy B,h}:M_h^r\to M_h^r$ by
\begin{equation}\label{Discrete_Elliptic}
(\Lambda_{\ssy B,h}\varphi,\chi)_{\ssy
0,D}:=(\partial^2\varphi,\partial^2\chi)_{\ssy 0,D}
+\mu\,(\partial^2\varphi,\chi)_{\ssy
0,D}, \quad\forall\,\varphi,\chi\in M_h^r,
\end{equation}
and
\begin{equation}\label{Discrete_Elliptic_Aux}
{\widetilde\Lambda}_{\ssy B,h}\varphi:=\Lambda_{\ssy
B,h}\varphi+\mu^2\,\varphi,\quad\forall\,\varphi\in M_h^r.
\end{equation}
Also, let $P_h:L^2(D)\to M_h^r$ be the usual $L^2(D)-$projection operator
onto $M_h^r$ for which it holds that
\begin{equation*}
(P_hf,\chi)_{\ssy 0,D}=(f,\chi)_{\ssy 0,D},\quad\forall\,\chi\in
M_h^r,\quad\forall\,f\in L^2(D).
\end{equation*}
%
%
\par
A finite element approximation ${\widetilde v}_{\ssy
B,h}\in M_h^r$ of the solution ${\widetilde v}_{\ssy B}$ of
\eqref{ElOp2_mu} is defined by the requirement
\begin{equation}\label{fem2_mu}
{\widetilde\Lambda}_{\ssy B,h}{\widetilde v}_{\ssy B,h}=P_hf,
\end{equation}
where the operator ${\widetilde\Lambda}_{\ssy B,h}$ is invertible since
\begin{equation}\label{mu_coercivity}
({\widetilde\Lambda}_{\ssy B,h}\chi,\chi)_{\ssy
0,D}\ge\,\tfrac{1}{2}\,\left(\,\|\partial^2\chi\|_{\ssy
0,D}^2+\mu^2\,\|\chi\|_{\ssy 0,D}^2\,\right), \quad\forall\chi\in
M_h^r.
\end{equation}
%
%
%
%
%
%
%
%
%
%
%
Thus, we denote by  ${\widetilde T}_{\ssy
B,h}:L^2(D)\to M_h^r$ the solution operator of \eqref{fem2_mu},
i.e.
\begin{equation*}
{\widetilde T}_{\ssy B,h}f:={\widetilde v}_{\ssy
B,h}={\widetilde\Lambda}_{\ssy B,h}^{-1}P_hf,\quad\forall\,f\in L^2(D).
\end{equation*}
Next, we derive an $L^2(D)$ error estimate for the finite element method \eqref{fem2_mu}.
%
%
\begin{proposition}\label{H2CASE}
Let $r\in\{2,3\}$. Then we have
\begin{equation}\label{ARA2}
\begin{split}
\|{\widetilde T}_{\ssy B}f-{\widetilde T}_{\ssy B,h}f\|_{\ssy 0,D} \leq\,C\left\{ \aligned
&h^4\,\|f\|_{\ssy 0,D},\quad r=3,\\
&h^3\,\|f\|_{\ssy -1,D},\hskip0.15truecm r=3,\\
&h^2\,\|f\|_{\ssy -1,D},\hskip0.15truecm r=2,\\
\endaligned\right.\quad\quad\forall\,f\in L^2(D),
\end{split}
\end{equation}
where $C$ is a positive constant independent of $h$ and $f$.
\end{proposition}
%
%
\begin{proof}
Let $f\in L^2(D)$, $e={\widetilde T}_{\ssy B}f-{\widetilde T}_{\ssy
B,h}f$ and ${\widetilde v}={\widetilde T}_{\ssy B}e$. To simplify the notation we
define ${\mathcal B}:H^2(D)\times H^2(D)\rightarrow{\mathbb R}$ by
${\mathcal B}(v,w):=(\partial^2v,\partial^2w)_{\ssy 0,D}+\mu\,(\partial^2v,w)_{\ssy 0,D}
+\mu^2\,(v,w)_{\ssy 0,D}$ for $v$, $w\in H^2(D)$. It is easily seen that
\begin{equation}\label{siriza12}
\gathered
{\mathcal B}(v,w)\leq\,\sqrt{2}\,(1+\mu)\,\left(\,\|\partial^2v\|^2_{\ssy 0,D}
+\mu^2\,\|v\|_{\ssy 0,D}^2\,\right)^{\frac{1}{2}}\,\|w\|_{\ssy 2,D}\quad\forall\,v,w\in H^2(D),\\
{\mathcal B}(v,v)\ge\,\tfrac{1}{2}\,\left[\,\|\partial^2v\|^2_{\ssy 0,D}
+\mu^2\,\|v\|_{\ssy 0,D}^2\,\right]\quad\forall\,v\in H^2(D).
\endgathered
\end{equation}
Later in the proof we shall use the symbol $C$ for a generic constant that
is independent of $h$ and $f$, and may changes value from one line to the other.
\par
First, we observe that $\|e\|_{\ssy 0,D}^2 ={\mathcal B}(e,{\widetilde v})$.
Then, we use the Galerkin orthogonality to get
\begin{equation*}
\|e\|_{\ssy 0,D}^2={\mathcal B}(e,{\widetilde v}-\chi),\quad\forall\,\chi\in M_h^r,
\end{equation*}
which, along with \eqref{siriza12}, leads to
\begin{equation}\label{ikariam_4}
\|e\|_{\ssy 0,D}^2\leq\,C\,\left(\,\|\partial^2e\|_{\ssy
0,D}^2+\mu^2\,\|e\|_{\ssy 0,D}^2\,\right)^{\frac{1}{2}}
\,\,\inf_{\chi\in M_h^r}\|{\widetilde v}-\chi\|_{\ssy 2,D}.
\end{equation}
Using again \eqref{siriza12} and the Galerkin orthogonality, we obtain
\begin{equation*}
\begin{split}
\|\partial^2e\|_{\ssy 0,D}^2+\mu^2
\,\|e\|_{\ssy 0,D}^2\leq&\,2\,{\mathcal B}(e,e)\\
\leq&\,2\,{\mathcal B}(e,{\widetilde T}_{\ssy B}f-\chi)\\
\leq&\,C\,\left(\,\|\partial^2e\|_{\ssy 0,D}^2
+\mu^2\,\|e\|^2_{\ssy 0,D}\,\right)^{\frac{1}{2}}\,
\|{\widetilde T}_{\ssy B}f-\chi\|_{\ssy 2,D},
\quad\forall\,\chi\in M_h^r,\\
\end{split}
\end{equation*}
which yields that
\begin{equation}\label{siriza13}
\left(\,\|\partial^2e\|_{\ssy 0,D}^2+\mu^2\,\|e\|_{\ssy 0,D}^2\,\right)^{\frac{1}{2}}
\leq\,C\,\inf_{\chi\in M_h^r}\|{\widetilde T}_{\ssy B}f-\chi\|_{\ssy 2,D}.
\end{equation}
Combining \eqref{ikariam_4}, \eqref{siriza13} and \eqref{Sh_H2},
we arrive at
\begin{equation}\label{siriza14}
\begin{split}
\|e\|_{\ssy 0,D}^2\leq&\,C\,\inf_{\chi\in M_h^r}
\|{\widetilde T}_{\ssy B}f-\chi\|_{\ssy 2,D}
\,\,\inf_{\chi\in M_h^r}\|{\widetilde v}-\chi\|_{\ssy 2,D}\\
\leq&\,C\,h^{s+s'-2}\,\|{\widetilde T}_{\ssy B}f\|_{\ssy
s+1,D}\,\|{\widetilde T}_{\ssy B}e\|_{\ssy s'+1,D},\quad
\forall\,s,s'\in\{2,r\}.
\end{split}
\end{equation}
\par
Let $r=2$. We use \eqref{siriza14} and \eqref{ElBihar2} to get
\begin{equation*}
\begin{split}
\|e\|_{\ssy 0,D}^2\leq&\,C\,h^2\,\|{\widetilde T}_{\ssy B}f\|_{\ssy
3,D}\,\|{\widetilde T}_{\ssy B}e\|_{\ssy 3,D}\\
\leq&\,C\,h^2\,\|f\|_{\ssy -1,D}\,\|e\|_{\ssy -1,D}\\
\leq&\,C\,h^2\,\|f\|_{\ssy -1,D}\,\|e\|_{\ssy 0,D},\\
\end{split}
\end{equation*} 
from which we conclude \eqref{ARA2} for $r=2$.
\par
Let $r=3$. We use \eqref{siriza14} with $s'=3$ and \eqref{ElBihar2} 
to obtain
\begin{equation*}
\begin{split}
\|e\|_{\ssy 0,D}^2\leq&\,C\,h^{s+1}\,\|{\widetilde T}_{\ssy B}f\|_{\ssy
s+1,D}\,\|{\widetilde T}_{\ssy B}e\|_{\ssy 4,D}\\
\leq&\,C\,h^{s+1}\,\|f\|_{\ssy s-3,D}\,\|e\|_{\ssy 0,D},\quad s=2,3,\\
\end{split}
\end{equation*} 
from which we conclude \eqref{ARA2} for $r=3$.
\end{proof}
%
%
%
Let ${\widetilde\gamma}_{\ssy B,h}:L^2(D)\times
L^2(D)\rightarrow{\mathbb R}$ be defined by
\begin{equation*}
{\widetilde\gamma}_{\ssy B,h}(f,g)=({\widetilde T}_{\ssy
B,h}f,g)_{\ssy 0,D}\quad\forall\,f, g\in L^2(D).
\end{equation*}
Then, as a simple consequence of \eqref{mu_coercivity}, the
following inequality holds
\begin{equation}\label{Out_of_Time2}
{\widetilde\gamma}_{\ssy B,h}(f,f)\ge\,\tfrac{1}{2}\,\left(
\|\partial^2({\widetilde T}_{\ssy B,h}f)\|_{\ssy
0,D}^2+\mu^2\,\|{\widetilde T}_{\ssy B,h}f\|_{\ssy
0,D}^2\,\right),\quad\forall\,f\in L^2(D).
\end{equation}
Thus, observing that
\begin{equation*}\label{adjo_Bh}
({\widetilde T}_{\ssy B,h}f,g)_{\ssy 0,D}=(f,{\widetilde
T}_{\ssy B,h}g)_{\ssy 0,D},\quad\forall\,f,g\in L^2(D),
\end{equation*}
and using \eqref{Out_of_Time2}, we easily conclude that
${\widetilde\gamma}_{\ssy B,h}$ is an inner product in $L^2(D)$.
We close this section with the following useful lemma.
%
%
\begin{lemma}
There exists a positive constant $C>0$ such that
\begin{equation}\label{PanwFragma1}
{\widetilde\gamma}_{\ssy B,h}(f,f)\leq\,C\,\|f\|^2_{\ssy
-2,D}, \quad\forall\,f\in L^2(D).
\end{equation}
\end{lemma}
%
%
%
%
\begin{proof}
Let $f\in L^2(D)$, $\psi={\widetilde T}_{\ssy B}f$ and
$\psi_h={\widetilde T}_{\ssy B,h}f$. Then, we have
\begin{equation}\label{zarkadi1}
\begin{split}
({\widetilde T}_{\ssy B,h}f,f)_{\ssy
0,D}=&\,({\widetilde\Lambda}_{\ssy
B}\psi,\psi_h)_{\ssy 0,D}\\
=&\,(\partial^2\psi,\partial^2\psi_h)_{\ssy 0,D}
+\mu\,(\partial^2\psi,\psi_h)_{\ssy 0,D}
+\mu^2\,(\psi,\psi_h)_{\ssy 0,D}\\
\leq&\,\tfrac{1}{\varepsilon}\,\left(\,\|\partial^2\psi\|_{\ssy
0,D}^2+\mu^2\,\|\psi\|^2_{\ssy
0,D}\,\right)+\varepsilon\,\left(\, \|\partial^2\psi_h\|_{\ssy 0,D}^2
+\mu^2\,\|\psi_h\|_{\ssy 0,D}^2\,\right),
\quad\forall\,\varepsilon>0.\\
\end{split}
\end{equation}
Setting $\varepsilon=\frac{1}{4}$ in \eqref{zarkadi1} and then
combining it with \eqref{Out_of_Time2}, we obtain
\begin{equation}\label{zarkadi3}
\|\partial^2\psi_h\|_{\ssy 0,D}^2+\mu^2\,\|\psi_h\|_{\ssy
0,D}^2\leq\,16\,\left(\,\|\partial^2\psi\|_{\ssy 0,D}^2
+\mu^2\,\|\psi\|_{\ssy 0,D}^2\,\right).
\end{equation}
Finally, \eqref{zarkadi1} with $\varepsilon=\tfrac{1}{4}$, \eqref{zarkadi3}
and \eqref{ElBihar2} yield
\begin{equation*}\label{zarkadi2}
\begin{split}
{\widetilde\gamma}_{\ssy B,h}(f,f)\leq&\,8\,\left(\,
\|\partial^2\psi\|_{\ssy
0,D}^2+\mu^2\,\|\psi\|_{\ssy 0,D}^2\,\right)\\
\leq&\,8\,(1+\mu^2)\,\|{\widetilde T}_{\ssy B}f\|_{\ssy 2,D}^2\\
\leq&\,8\,(1+\mu^2)\,C_{\ssy R,2}\,\|f\|_{\ssy -2,D}^2.\\
\end{split}
\end{equation*}
Thus, we arrived at \eqref{PanwFragma1}.
\end{proof}
%
%
%
%
%
%
\section{An Estimate for the Modeling Error}\label{SECTION2}
In this section, we estimate the modeling error in terms of
$\Delta{t}$ and $\Delta{x}$ (cf. Theorem~3.1 in \cite{KZ2010}).
%
%
\begin{theorem}\label{BIG_Qewrhma1}
Let $u$ be the solution of \eqref{PARAP} and ${\widehat u}$ be the
solution of \eqref{AC2}. Then, there exists a real constant ${\widetilde C}>0$,
independent of $\Delta{t}$ and $\Delta{x}$, such that
\begin{equation}\label{ModelError}
\max_{[0,T]}\left(\,{\mathbb E}\left[\|u-{\widehat u}\|_{\ssy
0,D}^2\right]\,\right)^{\half}
\leq\,{\widetilde C}\,\left[\,\omega_0(\Delta{t})\,\Delta{t}^\frac{1}{8}
+\epsilon^{-\frac{1}{2}}\,\Delta{x}^{\frac{1}{2}-\epsilon}\,\right],
\quad\forall\,\epsilon\in\left(0,\tfrac{1}{2}\right],
\end{equation}
where $\omega_0(\Delta{t}):=\sqrt{1+\Delta{t}^{\frac{3}{4}}}$.
\end{theorem}
%
%
%
%
%
%
%
%
%
%
%
%
\begin{proof}
Using \eqref{MildSol}, \eqref{HatUform} and Lemma~\ref{Lhmma1}, we
conclude that
\begin{equation}\label{corv0}
u(t,x)-{\widehat u}(t,x)=\int_0^{\ssy T}\!\!\!\int_{\ssy D}
\big[{\mathcal X}_{(0,t)}(s)\,\Psi(t-s;x,y) -{\widetilde
\Psi}(t,x;s,y)\big]\,dW(s,y), \quad\forall\,(t,x)\in[0,T]\times
{\overline D},
\end{equation}
where ${\widetilde{\Psi}}: (0,T)\times{D}\rightarrow
L^2((0,T)\times D)$ is given by
\begin{equation*}
{\widetilde \Psi}(t,x;s,y):=\tfrac{1}{\Delta{t}} \int_{\ssy T_n}
{\mathcal X}_{(0,t)}(s') \,\left[\,\sum_{i=1}^{\ssy
J_{\star}+1}\psi_i(y)\left(\,\sum_{\ell=1}^{\ssy
J_{\star}+1}\,G_{i,\ell}^{-1}\,\int_{\ssy
D}\Psi(t-s';x,y')\,\psi_{\ell}(y')\;dy'\,\right)\,\right]\,ds',
\quad\forall\,(s,y)\in T_n\times{D},
\end{equation*}
for $n=1,\dots,N_{\star}$.
\par
Let $\Theta:=\left\{{\mathbb E}\left[ \|u-{\widehat u}\|_{\ssy
0,D}^2\right]\right\}^{\half}$ and $t\in(0,T]$. Using
\eqref{corv0} and It{\^o} isometry \eqref{Ito_Isom}, we obtain
\begin{equation*}
\Theta(t)=\Bigg\{\int_0^{\ssy T}\int_{\ssy D}\!\int_{\ssy
D}\,\left[ {\mathcal X}_{(0,t)}(s)\,\Psi(t-s;x,y)-{\widetilde
\Psi}(t,x;s,y)\right]^2\;dxdyds \Bigg\}^{\frac{1}{2}}.
\end{equation*}
Now, we introduce the splitting
\begin{equation}\label{corv1}
\Theta(t) \leq\,\Theta_{\ssy A}(t)
+\Theta_{\ssy B}(t),
\end{equation}
where
\begin{equation*}
\Theta_{\ssy A}(t):=\left\{\sum_{n=1}^{\ssy N_{\star}}\int_{\ssy
D}\!\int_{\ssy D}\int_{\ssy
T_n}\left[\tfrac{1}{\Delta{t}}\,\int_{\ssy T_n}{\mathcal
X}_{(0,t)}(s')\,\Psi(t-s';x,y)\,ds'-{\widetilde
\Psi}(t,x;s,y)\right]^2\,dxdyds\right\}^{\frac{1}{2}}
\end{equation*}
and
\begin{equation*}
\Theta_{\ssy B}(t):=\left\{\sum_{n=1}^{\ssy N_{\star}}\int_{\ssy
D}\!\int_{\ssy D}\int_{\ssy T_n}\left[{\mathcal
X}_{(0,t)}(s)\,\Psi(t-s;x,y)-\tfrac{1}{\Delta{t}}\,\int_{\ssy
T_n}{\mathcal
X}_{(0,t)}(s')\,\Psi(t-s';x,y)\,ds'\right]^2\,dxdyds\right\}^{\frac{1}{2}}.
\end{equation*}
Also, to simplify the notation in the rest of the proof, we set
$\mu_k:=\lambda_k^2\,(\lambda_k^2-\mu)$ for $k\in{\mathbb N}$, and 
use the symbol $C$ to denote a generic constant that is independent
of $\Delta{t}$ and $\Delta{x}$ and may changes value from one line to the other.
%
%
\par
\textbullet\, {\tt Estimation of $\Theta_{\ssy A}(t)$}: Using
\eqref{GreenKernel} and the $(\cdot,\cdot)_{\ssy
0,D}-$orthogonality of $(\varepsilon_k)_{k=1}^{\infty}$, we have
\begin{equation*}
\begin{split}
\Theta^2_{\ssy A}(t) &=\,\tfrac{1}{\Delta{t}} \sum_{n=1}^{\ssy
N_{\star}}\int_{\ssy D}\int_{\ssy D} \Bigg[\int_{\ssy
T_n}{\mathcal X}_{(0,t)}(s')
\,\Big[\Psi(t-s';x,y)-\sum_{\ell,i=1}^{\ssy
J_{\star}+1}G_{i,\ell}^{-1}\,\left(\Psi(t-s';x,\cdot),
\psi_{\ell}(\cdot)\right)_{\ssy 0,D}
\,\psi_i(y)\Big]\,ds'\Bigg]^2\,dydx\\
&=\,\tfrac{1}{\Delta{t}} \sum_{n=1}^{\ssy N_{\star}}
\Bigg[\sum_{k=1}^{\infty}\,\left(\int_{\ssy T_n} {\mathcal
X}_{(0,t)}(s')\,e^{-\mu_k(t-s')}\;ds'\right)^2 \,\int_{\ssy
D}\Big(\varepsilon_k'(y)-\sum_{\ell,i=1}^{\ssy
J_{\star}+1}G^{-1}_{i,\ell}
\,(\varepsilon_k',\psi_{\ell})_{\ssy 0,D}\,\psi_i(y)\Big)^2\;dy\Bigg]\\
\end{split}
\end{equation*}
from which, using the Cauchy-Schwarz inequality, follows that
\begin{equation}\label{corv2}
\Theta^2_{\ssy A}(t)\leq\sum_{k=1}^{\kappa}A_k(t)\,B_k
+\sum_{k=\kappa+1}^{\infty}A_k(t)\,B_k,
\end{equation}
%
%
%
where
\begin{equation*}
\begin{split}
A_k(t):=&\,2\,\lambda_k^2\int_0^t e^{-2\mu_k(t-s')}\,ds',\\
B_k:=&\,\int_{\ssy D}\Big(\varphi_k(y)-\sum_{\ell,i=1}^{\ssy
J_{\star}+1}G^{-1}_{i,\ell}\,(\varphi_k,\psi_{\ell})_{\ssy
0,D}\,\psi_{i}(y)\Big)^2\;dy.\\
\end{split}
\end{equation*}
First, we observe that
\begin{equation}\label{2009b}
\begin{split}
\sqrt{B_k}\leq&\,\max_{1\leq{j}\leq{\ssy
J_{\star}}}\,\,\sup_{x,y\in{\ssy D_j}}
\left|\,\varphi_k(x)-\varphi_k(y)\,\right|\\
\leq&\,\min\{1,\,\lambda_k\Delta{x}\}\\
\leq&\,\min\left\{1,(\sqrt{2}\lambda_k\Delta{x})^{\theta}\right\},
\quad\forall\,\theta\in[0,1],\quad\forall\,k\in{\mathbb N}.\\
\end{split}
\end{equation}
Next, we use \eqref{mu_condition}, to obtain
\begin{equation}\label{2009a}
\begin{split}
A_k(t)\leq&\,\tfrac{1-e^{-2\mu_kt}}{\lambda_k^2-\mu}\\
<&\,\tfrac{(\kappa+1)^2}{1+2\,\kappa}
\,\tfrac{1}{\lambda_k^2},\quad\forall\,k\ge\kappa+1.\\
\end{split}
\end{equation}
Thus, from \eqref{corv2}, \eqref{2009b} and \eqref{2009a},
we conclude that
\begin{equation*}
\Theta_{\ssy A}^2(t)\leq\,C\,\left(\,(\Delta{x})^2
\,\sum_{k=1}^{\kappa}\lambda_k^2
+(\Delta{x})^{2\theta}
\,\sum_{k=\kappa+1}^{\infty}\tfrac{1}{\lambda_k^{2-2\theta}}\right)
\end{equation*}
which yields
\begin{equation}\label{corv5}
\Theta_{\ssy A}(t)\leq\,C\,(\Delta{x})^{\theta}\
\,\left(\,\sum_{k=1}^{\infty}\frac{1}{\lambda_k^{1+2(\frac{1}{2}-\theta)}}
\,\right)^{\frac{1}{2}},\quad\forall\,\theta\in[0,\tfrac{1}{2}).
\end{equation}
%
%
%
\par
\textbullet\,
{\tt Estimation of $\Theta_{\ssy B}(t)$}:
For $t\in (0,T]$, let
${\widehat N}(t):= \min\big\{\,\ell\in\nset:\ken 1\leq \ell\leq
N_{\star} \ken\text{\rm and}\ken t\leq t_{\ell}\,\big\}$
and
\begin{equation*}
{\widehat T}_n(t):=T_n\cap (0,t)=\left\{
\begin{aligned}
&T_n,\quad\quad\quad\quad\text{\rm if}\ken n<{\widehat N}(t)\\
&(t_{\ssy {\widehat N}(t)-1},t),\hskip0.31truecm
\text{\rm if}\ken n={\widehat N}(t)\\
\end{aligned}
\right.,\quad n=1,\dots,{\widehat N}(t).
\end{equation*}
Thus, using \eqref{GreenKernel} and the $(\cdot,\cdot)_{\ssy
0,D}-$orthogonality of $(\varepsilon_k)_{k=1}^{\infty}$ and
$(\varphi_k)_{k=1}^{\infty}$ as follows
\begin{equation*}
\begin{split}
\Theta^2_{\ssy B}(t)&= \tfrac{1}{(\Delta{t})^2}
\sum_{n=1}^{N_{\star}} \int_{\ssy D}\!\int_{\ssy D}\int_{\ssy T_n}
\Bigg[\int_{\ssy T_n}\Big[{\mathcal X}_{(0,t)}(s)\,\Psi(t-s;x,y)
-{\mathcal X}_{(0,t)}(s')\,\Psi(t-s';x,y)\Big]\,ds'
\Bigg]^2\,dxdyds\\
&= \tfrac{1}{(\Delta{t})^2} \sum_{n=1}^{N_{\star}} \int_{\ssy
D}\!\int_{\ssy D}\int_{\ssy T_n}
\Bigg[\sum_{k=1}^{\infty}\lambda_k\,\varepsilon_k(x)\,\varphi_k(y)\int_{\ssy
T_n}\Big[{\mathcal X}_{(0,t)}(s)\,e^{-\mu_k(t-s)} -{\mathcal
X}_{(0,t)}(s')\,e^{-\mu_k(t-s')}\Big]\,ds'
\Bigg]^2\,dxdyds\\
\end{split}
\end{equation*}
we conclude that
\begin{equation}\label{corv6}
\Theta^2_{\ssy B}(t)\leq\,
\sum_{k=1}^{\infty}\,\lambda_k^2\,\left(\,\tfrac{1}{(\Delta{t})^2}\,
\sum_{n=1}^{{\widehat N}(t)}\Psi_n^k(t)\,\right),
\end{equation}
where
\begin{equation*}
\Psi_n^k(t):=\int_{\ssy T_n} \left[\,\int_{\ssy T_n}\left(
{\mathcal X}_{(0,t)}(s)\,e^{-\mu_k(t-s)} -{\mathcal
X}_{(0,t)}(s')\,e^{-\mu_k(t-s')} \right)\,ds'\,\right]^2\,ds.
\end{equation*}
Let $k\in{\mathbb N}$ and $n\in\{1,\dots,{\widehat N}(t)-1\}$.
Then, we have
\begin{equation*}\label{siriza1}
\begin{split}
\Psi_n^k(t)&=\,\int_{\ssy T_n}\Big(\,\int_{\ssy T_n}
\!\!\int_s^{s'} \mu_k\,
e^{-\mu_k(t-\tau)}\,d\tau ds'\,\Big)^2\,ds\\
&\leq\,\int_{\ssy T_n}
\Big(\,\int_{\ssy T_n}\!\!\int_{t_{n-1}}^{\max\{s',s\}}
\mu_k\,e^{-\mu_k(t-\tau)}\,d\tau ds'\,\Big)^2\,ds\\
&\leq\,2\int_{\ssy T_n}\Big(\, \int_{\ssy
T_n}\!\!\int_{t_{n-1}}^{s'} \mu_k\,e^{-\mu_k(t-\tau)}\,d\tau
ds'\,\Big)^2\,ds +2\int_{\ssy T_n}\Big(\, \int_{\ssy
T_n}\!\!\int_{t_{n-1}}^s\mu_k\,
e^{-\mu_k(t-\tau)}\,d\tau\,ds'\,\Big)^2\,ds\\
&\leq\,2\,\Delta{t}\,\Big(\, \int_{\ssy
T_n}\!\!\int_{t_{n-1}}^{s'} \mu_k\, e^{-\mu_k(t-\tau)}\,d\tau
ds'\Big)^2 +2\,(\Delta{t})^2\,\int_{\ssy
T_n}\Big(\int_{t_{n-1}}^{s} \mu_k\,
e^{-\mu_k(t-\tau)}\,d\tau\Big)^2\,ds,\\
\end{split}
\end{equation*}
from which, after using the Cauchy-Schwarz inequality, we arrive at
\begin{equation}\label{siriza1}
\Psi_n^k(t)\leq\,4\,(\Delta{t})^2\,\int_{\ssy
T_n}\Big(\int_{t_{n-1}}^{s} \mu_k\,
e^{-\mu_k(t-\tau)}\,d\tau\Big)^2\,ds.
\end{equation}
For $k\leq\kappa$, we use \eqref{siriza1} to get
\begin{equation}\label{siriza2}
\Psi_n^k(t)\leq\,4\,\max_{1\leq{k}\leq\kappa}(\mu_k)^2\,(\Delta{t})^5.
\end{equation}
For $k\ge\kappa+1$, we use \eqref{siriza1} to have
\begin{equation}\label{siriza3}
\begin{split}
\Psi_n^k(t)&\leq\,4\,(\Delta{t})^2\,\int_{\ssy T_n}
\Big(\,e^{-\mu_k(t-s)}-e^{-\mu_k(t-t_{n-1})}\,\Big)^2\,ds\\
&\leq\,4\,(\Delta{t})^2\,\big(1-e^{-\mu_k\Delta{t}}\big)^2
\int_{\ssy T_n}e^{-2\mu_k(t-s)}\,ds\\
&\leq\,2\,(\Delta{t})^2\,\big(1-e^{-\mu_k\Delta{t}}\big)^2
\,\tfrac{e^{-\mu_k(t-t_n)}
-e^{-\mu_k(t-t_{n-1})}}{\mu_k}\cdot\\
\end{split}
\end{equation}
Summing with respect to $n$, and using \eqref{siriza1}, \eqref{siriza2} and \eqref{siriza3},
we obtain
\begin{equation}\label{corv7}
\tfrac{1}{(\Delta{t})^2}\,\sum_{n=1}^{{\widehat
N}(t)-1}\Psi_n^k(t)\leq\,C\,\left\{
\aligned
&(\Delta{t})^2,\hskip1.2truecm k\leq\kappa,\\
&\tfrac{(1-e^{-\mu_k\Delta{t}})^2}{\mu_k},\quad k\ge\kappa+1\\
\endaligned\right.\cdot
\end{equation}
Considering, now, the case $n={\widehat N}(t)$, we have
\begin{equation}\label{gaga}
\Psi_{\ssy {\widehat N}(t)}^k(t)=\Psi^k_{\ssy A}(t)+\Psi_{\ssy B}^k(t)
\end{equation}
with
\begin{equation*}
\begin{split}
\Psi_{\ssy A}^k(t)&:=\int_{t_{{\widehat N}(t)-1}}^t
\left(\,\int_{t_{{\widehat N}(t)-1}}^t
\int_{s'}^s\mu_k\,e^{-\mu_k(t-\tau)}\,d\tau{ds'}
+\int_t^{t_{{\widehat N}(t)}} e^{-\mu_k(t-s)}\,ds'
\,\right)^2\,ds,\\
\Psi_{\ssy B}^k(t)&:=\int_t^{t_{{\widehat N}(t)}}\left(\,
\int_{t_{{\widehat N}(t)}-1}^t
e^{-\mu_k(t-s')}\,ds'\,\right)^2\,ds.\\
\end{split}
\end{equation*}
For $k\leq\kappa$, we obtain
\begin{equation}\label{siriza4}
\tfrac{1}{(\Delta{t})^2}\,\Psi_{\ssy {\widehat N}(t)}^k(t)\leq\,C\,\Delta{t}.
\end{equation}
For $k\ge\kappa+1$, we have
\begin{equation*}
\begin{split}
\Psi_{\ssy B}^k(t)&\leq\,\tfrac{\Delta{t}}{\mu_k^2}\,
\Big[\,1-e^{-\mu_k\,
\big(\,t-t_{{\widehat N}(t)-1}\,\big)}\,\Big]^2\\
&\leq\,\tfrac{\Delta{t}}{\mu_k^2}\,
\big(\,1-e^{-\mu_k\,\Delta{t}}\,)^2\\
\end{split}
\end{equation*}
and
\begin{equation*}
\begin{split}
\Psi_{\ssy A}^k(t)&\leq\,\int_{t_{{\widehat N}(t)-1}}^t \left[
\int_{t_{{\widehat N}(t)-1}}^t
\int_{s'}^s\mu_k\,e^{-\mu_k(t-\tau)}\,d\tau{ds'}
+\Delta{t}\,\,\,e^{-\mu_k(t-s)}
\right]^2\,ds \\
&\leq\,2\,\int_{t_{{\widehat N}(t)-1}}^t \left[ \int_{t_{{\widehat
N}(t)-1}}^t\int_{s'}^s \mu_k\,e^{-\mu_k(t-\tau)}\,d\tau{ds'}
\right]^2\,ds +\tfrac{(\Delta{t})^2}{\mu_k}\,
\left[\,1-e^{-2\mu_k\left(\,t-t_{{\widehat N}(t)-1}\,\right)}\,\right]\\
&\leq\, 2\,\int_{t_{{\widehat N}(t)-1}}^t \left[
\int_{t_{{\widehat N}(t)-1}}^t\int_{t_{{\widehat
N}(t)-1}}^{\max\{s,s'\}} \mu_k\,e^{-\mu_k(t-\tau)}\,d\tau{ds'}
\right]^2\,ds
+\tfrac{(\Delta{t})^2}{\mu_k}\,\big(\,1-e^{-2\mu_k\,\Delta{t}}\,\big)\\
&\leq \,8\,(\Delta{t})^2 \int_{t_{{\widehat N}(t)-1}}^t
\left[\,\int_{t_{{\widehat N}(t)-1}}^s\mu_k\,
e^{-\mu_k(t-\tau)}\,d\tau\,\right]^2\,ds
+\tfrac{(\Delta{t})^2}{\mu_k}\,\big(\,1-e^{-2\mu_k\,\Delta{t}}\,\big)\\
&\leq \,8\,(\Delta{t})^2\, \int_{t_{{\widehat N}(t)-1}}^t
\left[\,e^{-\mu_k(t-s)} -e^{-\mu_k(t-t_{{\widehat
N}(t)-1})}\,\right]^2\,ds
+\tfrac{(\Delta{t})^2}{\mu_k}\,\big(\,1-e^{-2\mu_k\,\Delta{t}}\,\big),\\
\end{split}
\end{equation*}
which, along with \eqref{gaga}, gives
\begin{equation*}
\Psi_{\ssy {\widehat N}(t)}^k\leq\,5\tfrac{(\Delta{t})^2}{\mu_k}
\,\big(\,1-e^{-2\mu_k\,\Delta{t}}\,\big)
+\tfrac{\Delta{t}}{\mu_k^2}\,
\big(\,1-e^{-\mu_k\Delta{t}}\,\big)^2\,\cdot
\end{equation*}
Since the mean value theorem yields:
$1-e^{-\mu_k\Delta{t}}\leq\,\mu_k\,\Delta{t}$, the above
inequality takes the form
\begin{equation}\label{corv8}
\tfrac{1}{(\Delta{t})^2}\,\Psi_{\ssy {\widehat N}(t)}^k\leq
\,6\,\tfrac{1-e^{-2\mu_k\,\Delta{t}}}{\mu_k}\,\cdot
\end{equation}
\par
Combining \eqref{corv6}, \eqref{corv7}, \eqref{siriza4} and \eqref{corv8} we obtain
\begin{equation}\label{corv9}
\begin{split}
\Theta_{\ssy B}^2(t)\leq&\,C\,\left[\Delta{t}+\sum_{k=\kappa+1}^{\infty}\lambda_k^2\,
\tfrac{1-e^{-2\,\mu_k\,\Delta{t}}}{\mu_k}\right]\\
\leq&\,C\,\left[\Delta{t}+\sum_{k=1}^{\infty}
\tfrac{1-e^{-c_0\,\lambda_k^4\,\Delta{t}}}{\lambda_k^2}\right],
\end{split}
\end{equation}
with $c_0=\frac{2\,(1+2\kappa)}{(\kappa+1)^2}$. To get a convergence estimate we have to
exploit the way the series depends on $\Delta{t}$ in the above
relation:
\begin{equation}\label{corv10}
\begin{split}
\sum_{k=1}^{\infty}
\tfrac{1-e^{-c_0\,\lambda_k^4\,\Delta{t}}}{\lambda_k^2}
%
&\leq\, \tfrac{1-e^{-c_0\,\pi^4\,\Delta{t}}}{\pi^2}
+\int_1^{\infty} \tfrac{1-e^{-c_0\,x^4\,\pi^4\,\Delta{t}}}
{x^2\,\pi^2}\,dx\\
&\leq\,C\,\left[\,\big(\,1-e^{-c_0\,\pi^4\,\Delta{t}}\,\big)
+\Delta{t}\,\int_1^{\infty}
x^2\,e^{-c_0\,x^4\,\pi^4\,\Delta{t}}\,dx\,\right]\\
&\leq\,C\,\left[\,\Delta{t} +(\Delta{t})^{\frac{1}{4}}\,
\int_0^{\infty} y^2\,e^{-2y^4}\,dy\,\right]\\
&\leq\,C\,\left[\,(\Delta{t})^{\frac{3}{4}}
+1\,\right]\,(\Delta{t})^{\frac{1}{4}}.\\
\end{split}
\end{equation}
Using the bounds \eqref{corv9} and \eqref{corv10} we conclude that
\begin{equation}\label{corv12}
\Theta_{\ssy B}(t)\leq \,C\,\left[\,(\Delta{t})^{\frac{3}{4}}
+1\,\right]^{\half}\,\Delta{t}^{\frac{1}{8}}.
\end{equation}
\par
The error bound \eqref{ModelError} follows by observing that
$\Theta(0)=0$ and combining the bounds \eqref{corv1},
\eqref{corv5}, \eqref{corv12} and \eqref{SR_BOUND}.
\end{proof}
%
%
%
%
%
%
\section{Time-Discrete Approximations}\label{SECTION3}
The Backward Euler time-stepping method for problem
\eqref{AC2} specifies an approximation ${\widehat U}^m$ of
${\widehat u}(\tau_m,\cdot)$ starting by setting
\begin{equation}\label{BackE1}
{\widehat U}^0:=0,
\end{equation}
and then, for $m=1,\dots,M$, by finding ${\widehat U}^m\in {\bfdot
H}^4(D)$ such that
\begin{equation}\label{BackE2}
{\widehat U}^m-{\widehat U}^{m-1}+\Delta\tau\,\Lambda_{\ssy
B}{\widehat U}^m=\int_{\ssy\Delta_m}\partial_x{\widehat
W}\,ds\quad\text{\rm a.s.}.
\end{equation}
The method is well-defined when the differential operator
$Q_{\ssy B,\Delta\tau}:=I+\Delta\tau\,\Lambda_{\ssy B}:{\bfdot H}^4(D)\rightarrow L^2(D)$
is invertible. It is easily seen that $Q_{\ssy B,\Delta\tau}$ is invertible when
$1+\Delta\tau\,\lambda_k^2\,(\lambda_k^2-\mu)\not=0$
for $k\in{\mathbb N}$, or equivalently when: $\kappa=1$ or $\kappa\ge2$ and 
$\Delta\tau\,\max\limits_{1\leq{k}\leq{\kappa-1}}\lambda_{k}^2 \,(\mu-\lambda_k^2)\not=1$.
If $\kappa\ge2$, then it is easily seen that $\max\limits_{1\leq{k}\leq{\kappa-1}}\lambda_{k}^2 \,(\mu-\lambda_k^2)\leq\frac{\mu^2}{4}$, so the condition $\Delta\tau\,\frac{\mu^2}{4}<1$ is a sufficient
condition for the invertibility of $Q_{\ssy B,\Delta\tau}$.
\subsection{The Deterministic Case}
The Backward Euler time-discrete approximations of the solution $w$ to
the deterministic problem \eqref{Det_Parab} are defined as follows: first
we set
\begin{equation}\label{BEDet1}
W^0:=w_0,
\end{equation}
and then, for $m=1,\dots,M$, we find $W^m\in{\bfdot H}^4(D)$
such that
\begin{equation}\label{BEDet2}
W^m-W^{m-1}+\Delta\tau\,\Lambda_{\ssy B}W^m=0.
\end{equation}
Obviously, the Backward Euler  time-discrete approximations are well-defined
when $Q_{\ssy B,\Delta\tau}$ is invertible. Our next step, is to derive an error
estimate in a discrete in time $L_t^2(L_x^2)$
norm, taking into account that, in constrast to the case $\mu=0$ considered in \cite{KZ2010},
the operator $\Lambda_{\ssy B}$ is not always invertible.  
%
%
%
\begin{proposition}\label{DetPropo1}
Let $(W^m)_{m=0}^{\ssy M}$ be the Backward Euler time-discrete
approximations of the solution $w$ of the problem
\eqref{Det_Parab} defined in \eqref{BEDet1}--\eqref{BEDet2}. Also,
we assume that $\kappa=1$, or $\kappa\ge2$ and
$\Delta\tau\,\mu^2<\frac{1}{4}$.
Then, there exists a constant $C>0$, independent of $\Delta\tau$,
such that
\begin{equation}\label{Ydaspis900}
\Bigg(\,\sum_{m=1}^{\ssy M}\Delta\tau\,
\|W^m-w(\tau_m,\cdot)\|_{\ssy 0,D}^2 \,\Bigg)^{\frac{1}{2}}
\leq\,C\,(\Delta\tau)^{\theta} \,\|w_0\|_{\ssy{\bfdot
H}^{4\theta-2}},\quad\forall\,w_0\in{\bfdot H}^2(D),\quad\forall\,\theta\in[0,1].
\end{equation}
\end{proposition}
%
%
%
%
%
%
%
%
\begin{proof}
The estimate \eqref{Ydaspis900}
will be established by interpolation, after proving it for $\theta=1$ and $\theta=0$.
\par
Let $w_0\in{\bfdot H}^2(D)$. According to the discussion in the
begining of this section, when $\kappa=1$ or $\kappa\ge2$
and $\Delta\tau\,\mu^2<\frac{1}{4}$, the existence and uniqueness of the
time-discrete approximations $(W^m)_{m=0}^{\ssy M}$ is secured. We omit the
case $\kappa=1$ since then the operator $\Lambda_{\ssy B}$ is invertible and the
proof of \eqref{Ydaspis900} follows moving along the lines of the proof of
Proposition~4.1 in \cite{KZ2010}, or
alternatively moving along the lines of the proof below using the operator $T_{\ssy B}$
instead of ${\widetilde T}_{\ssy B}$. Here, we will proceed with the proof of \eqref{Ydaspis900}
under the assumption $\Delta\tau\,\mu^2<\frac{1}{4}$, without using somewhere a possible
invertibilty of $\Lambda_{\ssy B}$. In the sequel,  we will use the symbol $C$ to denote a generic
constant that is independent of $\Delta{t}$ and may changes value from one line to the other. 
\par
Let $E^m(\cdot):=w(\tau_m,\cdot)-V^m(\cdot)$
for $m=0,\dots,M$ and $\sigma_{m}:=\,\int_{\ssy\Delta_m}
\left[\,w(\tau_m,\cdot)-w(\tau,\cdot)\,\right]\,d\tau$ for $m=1,\dots,M$.
Then, combining \eqref{Det_Parab} and \eqref{BEDet2}, we conclude
that
\begin{equation}\label{Axliopitas1}
{\widetilde T}_{\ssy B}(E^m-E^{m-1})
+\Delta\tau\,E^m=\Delta\tau\,\mu^2\,{\widetilde T}_{\ssy B}E^m+\left(\sigma_m
-\mu^2\,{\widetilde T}_{\ssy B}\sigma_m\right),
\quad m=1,\dots,M.\\
\end{equation}
Now, take the $L^2(D)-$inner product with $E^m$ of both sides of \eqref{Axliopitas1},
to obtain
\begin{equation}\label{axiografo_1}
\begin{split}
{\widetilde\gamma}_{\ssy B}(E^m
-E^{m-1},{E}^m)_{\ssy
0,D}+\Delta\tau\,\|E^m\|_{\ssy 0,D}^2=&\,
\Delta\tau\,\mu^2\,{\widetilde\gamma}_{\ssy B}(E^m,E^m)\\
&+(\sigma_m-\mu^2\,{\widetilde T}_{\ssy B}\sigma_m,E^m)_{\ssy
0,D},\quad m=1,\dots,M.\\
\end{split}
\end{equation}
%
%
%
Using \eqref{innerproduct}, \eqref{axiografo_1} and \eqref{ElBihar2}, we arrive at
\begin{equation}\label{axiografo_2}
\begin{split}
{\widetilde\gamma}_{\ssy B}({E}^m,{E}^m)
-{\widetilde\gamma}_{\ssy B}({E}^{m-1},{E}^{m-1})
+\Delta\tau\,\|{E}^m\|_{\ssy
0,D}^2\leq&\,2\,\Delta\tau\,\mu^2\,{\widetilde\gamma}_{\ssy B}(E^m,E^m)\\
&+C\,\Delta\tau^{-1}
\,\|\sigma_m\|_{\ssy 0,D}^2,\quad m=1,\dots,M.\\
\end{split}
\end{equation}
Since $2\,\Delta\tau\,\mu^2<1$, \eqref{axiografo_2} yields
\begin{equation*}\label{kill_bill1}
{\widetilde\gamma}_{\ssy B}(E^m,E^m)\leq\tfrac{1}{1-2\,\mu^2\,\Delta\tau}\,\left[\,
{\widetilde\gamma}_{\ssy B}(E^{m-1},E^{m-1})
+C\,\Delta\tau^{-1}\,\|\sigma_m\|_{\ssy 0,D}^2\right],
\quad m=1,\dots,M.
\end{equation*}
Then, we apply a simple induction argument and use that  $E^0=0$ 
and $4\,\Delta\tau\,\mu^2<1$, to obtain
\begin{equation}\label{kill_bill1}
\begin{split}
{\widetilde\gamma}_{\ssy B}(E^{m},E^{m})
\leq&\,C\,\Delta\tau^{-1}\,\sum_{\ell=1}^{m}\|\sigma_{\ell}\|_{\ssy 0,D}^2
\,\tfrac{1}{(1-2\,\Delta\tau\,\mu^2)^{m+1-\ell}}\\
\leq&\,C\,e^{4T\mu^2}
\,\Delta\tau^{-1}\,\sum_{\ell=1}^{m}\|\sigma_{\ell}\|_{\ssy 0,D}^2,
\quad m=1,\dots,M.\\
\end{split}
\end{equation}
Next, we use the Cauchy-Schwarz inequality to bound $\sigma_{m}$ as
follows:
\begin{equation}\label{Ydaspis902}
\begin{split}
\|\sigma_m\|_{\ssy 0,D}^2
%
&\leq\,C\,\int_{\ssy D}\left(\int_{\ssy\Delta_m}
\!\!\int_{\ssy\Delta_m}|\partial_{\tau}w(s,x)|\,ds
d\tau\right)^2\,dx\\
&\leq\,C\,(\Delta\tau)^3\,\int_{\ssy\Delta_m}
\|\partial_{\tau}w(s,\cdot)\|_{\ssy 0,D}^2\,ds,
\quad m=1,\dots,M.\\
\end{split}
\end{equation}
Thus, \eqref{Ydaspis902} and \eqref{kill_bill1} yield
\begin{equation}\label{kill_bill2}
{\widetilde\gamma}_{\ssy B}(E^{m},E^{m})
\leq\,C\,(\Delta\tau)^{2}\,\int_0^{\tau_m}\|\partial_{\tau}w(s,\cdot)\|_{\ssy 0,D}^2\;ds,
\quad m=1,\dots,M.
\end{equation}
Combining \eqref{axiografo_2}, \eqref{kill_bill2} and \eqref{Ydaspis902},
we have
\begin{equation}\label{kill_bill3}
\begin{split}
{\widetilde\gamma}_{\ssy B}({E}^m,{E}^m)
-{\widetilde\gamma}_{\ssy B}({E}^{m-1},{E}^{m-1})
+\Delta\tau\,\|{E}^m\|_{\ssy 0,D}^2\leq&\,C\,(\Delta\tau)^{2}
\,\int_{\ssy\Delta_m}
\|\partial_{\tau}w(s,\cdot)\|_{\ssy 0,D}^2\;ds\\
&+C\,(\Delta\tau)^3\,\int_0^{\tau_m}
\|\partial_{\tau}w(s,\cdot)\|_{\ssy 0,D}^2\;ds\\
\end{split}
\end{equation}
for $m=1,\dots,M$.
Summing with respect to $m$ from 1 up to $M$ and using the fact
that $E^0=0$, \eqref{kill_bill3} yields
\begin{equation}\label{axiografo_6}
{\widetilde\gamma}_{\ssy B}(E^{\ssy M},E^{\ssy M})
+\sum_{m=1}^{\ssy M}\Delta\tau\,\|E^m\|_{\ssy
0,D}^2\leq\,C\,(\Delta\tau)^{2}\,\int_0^{\ssy T}
\|\partial_{\tau}w(s,\cdot)\|_{\ssy 0,D}^2\;ds.
\end{equation}
Finally, use \eqref{axiografo_6}  and
\eqref{Reggo3} (with $\beta=0$, $\ell=1$, $p=0$) to obtain
\begin{equation}\label{axiografo_5}
\left(\,\sum_{m=1}^{\ssy M}\Delta\tau\,\,\|E^m\|_{\ssy
0,D}^2\,\right)^{\frac{1}{2}}
\leq\,C\,\Delta\tau\,\|w_0\|_{\ssy{\bfdot H}^2},\\
\end{equation}
which establishes \eqref{Ydaspis900} for $\theta=1$.
\par
First, we observe that \eqref{BEDet2} is written equivalently as
\begin{equation*}
{\widetilde T}_{\ssy B}(W^m-W^{m-1})+\Delta\tau\,W^m
=\Delta\tau\,\mu^2\,{\widetilde T}_{\ssy B}W^m, 
\quad m=1,\dots,M,
\end{equation*}
from which, after taking  the $L^2(D)-$inner product with $W^m$, we obtain
\begin{equation}\label{eklogesII2}
{\widetilde\gamma}_{\ssy B}(W^m-W^{m-1},W^m)_{\ssy 0,D}
+\Delta\tau\,\|W^m\|_{\ssy 0,D}^2=\Delta\tau\,\mu^2
\,{\widetilde\gamma}_{\ssy B}(W^m\,W^m),\quad m=1,\dots,M.
\end{equation}
Then, we combine \eqref{innerproduct} and \eqref{eklogesII2} to have
\begin{equation}\label{axiografo_10}
(1-2\,\Delta\tau\,\mu^2)\,{\widetilde\gamma}_{\ssy B}(W^m,W^m)
+2\,\Delta\tau\,\|W^m\|_{\ssy 0,D}^2\leq\,{\widetilde\gamma}_{\ssy
B}(W^{m-1},W^{m-1}),\quad m=1,\dots,M.
\end{equation}
Since  $4\,\mu^2\,\Delta\tau<1$, \eqref{axiografo_10} yields that
\begin{equation*}
\begin{split}
{\widetilde\gamma}_{\ssy B}(W^m,W^m)\leq&\,\tfrac{1}{1-2\,\mu^2\,\Delta\tau}
\,{\widetilde\gamma}_{\ssy B}(W^{m-1},W^{m-1})\\
\leq&\,e^{4\mu^2\Delta\tau}
\,{\widetilde\gamma}_{\ssy B}(W^{m-1},W^{m-1}),
\quad m=1,\dots,M,\\
\end{split}
\end{equation*}
from which, applying a simple induction argument, we conclude that
\begin{equation}\label{eklogesII3}
\max_{0\leq{m}\leq{\ssy M}}{\widetilde\gamma}_{\ssy B}(W^{m},W^{m})
\leq\,C\,{\widetilde\gamma}_{\ssy B}(w_0,w_0).
\end{equation}
Now, summing with respect to $m$ from $1$ up to $M$, and using \eqref{eklogesII3},
\eqref{axiografo_10} yields
\begin{equation}\label{axiografo_13}
\begin{split}
\sum_{m=1}^{\ssy M}\Delta\tau\,\|W^m\|_{\ssy
0,D}^2\leq&\,C\,({\widetilde T}_{\ssy B}w_0,w_0)_{\ssy 0,D}\\
\leq&\,\|w_0\|_{\ssy -2,D}\,\|{\widetilde T}_{\ssy B}w_0\|_{\ssy 2,D}.\\
\end{split}
\end{equation}
Thus, using \eqref{axiografo_13}, \eqref{ElBihar2} and
\eqref{minus_equiv}, we obtain
\begin{equation}\label{Ydaspis904}
\begin{split}
\left(\,\sum_{m=1}^{\ssy M}\Delta\tau\,\|W^m\|_{\ssy
0,D}^2\,\right)^{\frac{1}{2}}
&\leq\,C\,\|w_0\|_{\ssy -2,D}\\
&\leq\,C\,\|w_0\|_{\ssy{\bfdot H}^{-2}}.\\
\end{split}
\end{equation}
In addition we have
\begin{equation*}
\begin{split}
\sum_{m=1}^{\ssy M}\Delta\tau\,\|w(\tau_m,\cdot)\|_{\ssy 0,D}^2
&\leq\,\sum_{m=1}^{\ssy M}\int_{\ssy D}\,
\left(\int_{\ssy\Delta_m}\partial_{\tau}\big[(\tau-\tau_{m-1})
\,w^2(\tau,x)\big]\,d\tau\right)\,dx\\
&\leq\,\sum_{m=1}^{\ssy M}\int_{\ssy D}\,
\left(\int_{\ssy\Delta_m}\big[w^2(\tau,x)
+2\,(\tau-\tau_{m-1})\,w_{\tau}(\tau,x)\,w(\tau,x)\big]
\,d\tau\right)\,dx\\
&\leq\,\sum_{m=1}^{\ssy M}\int_{\ssy\Delta_m}\,
\Big[\,2\,\|w(\tau,\cdot)\|_{\ssy 0,D}^2
+(\tau-\tau_{m-1})^2\,
\|w_{\tau}(\tau,\cdot)\|_{\ssy 0,D}^2\Big]
\,d\tau\\
&\leq\,2\,\int_0^{\ssy T}\Big[\|w(\tau,\cdot)\|_{\ssy 0,D}^2
+\tau^2\,\|w_{\tau}(\tau,\cdot)\|_{\ssy 0,D}^2\Big]\,d\tau,
\end{split}
\end{equation*}
which, along with \eqref{Reggo3} (taking $(\beta,\ell,p)=(0,0,0)$
and $(\beta,\ell,p)=(2,1,0)$) and \eqref{minus_equiv},
yields
\begin{equation}\label{Ydaspis906}
\left(\,\sum_{m=1}^{\ssy M}\Delta\tau\,\|w(\tau_m,\cdot)\|_{\ssy
0,D}^2\,\right)^{\frac{1}{2}}\leq\,C\,\|w_0\|_{\ssy {\bfdot
H}^{-2}}.
\end{equation}
Thus, the estimate \eqref{Ydaspis900} for $\theta=0$ follows
easily combining \eqref{Ydaspis904} and \eqref{Ydaspis906}.
\end{proof}
%
%
%
%
\subsection{The Stochastic Case}
Next theorem combines the convergence result of
Proposition~\ref{DetPropo1} with a discrete Duhamel's principle in
order to prove a discrete in time $L^{\infty}_t(L^2_{\ssy
P}(L^2_x))$ convergence estimate for the time discrete
approximations of ${\widehat u}$ (cf. \cite{KZ2010}, \cite{YubinY05}).
%
%
%
%
\begin{theorem}\label{TimeDiscreteErr1}
Let ${\widehat u}$ be the solution of \eqref{AC2} and $({\widehat
U}^m)_{m=0}^{\ssy M}$ be the  time-discrete approximations
defined by \eqref{BackE1}--\eqref{BackE2}. Also, we assume that
$\kappa=1$, or $\kappa\ge2$ and $\Delta\tau\,\mu^2<\frac{1}{4}$.
Then, there exists a constant $C>0$, independent of $\Delta{t}$,
$\Delta{x}$ and $\Delta\tau$, such that
\begin{equation}\label{ElPasso}
\max_{1\leq m \leq {\ssy M}} \left({\mathbb E}\left[ \|{\widehat
U}^m-{\widehat u}(\tau_m,\cdot)\|_{\ssy
0,D}^2\right]\right)^{\half} \leq\,C
\,\omega_1(\Delta\tau,\epsilon)\,
\,\Delta{\tau}^{\frac{1}{8}-\epsilon},
\quad\forall\,\epsilon\in(0,\tfrac{1}{8}],
\end{equation}
where $\omega_1(\Delta\tau,\epsilon):=\epsilon^{-\frac{1}{2}}
+(\Delta\tau)^{\epsilon}(1+(\Delta{\tau})^{\frac{7}{4}}
+(\Delta{\tau})^{\frac{3}{4}})^{\half}$.
\end{theorem}
%
%
%
%
%
%
%
%
%
\begin{proof}
Let $I:L^2(D)\to L^2(D)$ be the identity operator,
$\Lambda:L^2(D)\to{\bfdot H}^4(D)$ be the inverse elliptic
operator $\Lambda:=(I+\Delta{\tau}\,\Lambda_{\ssy B})^{-1}$
which has Green function $G_{\ssy\Lambda}(x,y)=
\sum_{k=1}^{\infty}\frac{\varepsilon_k(x)
\,\varepsilon_k(y)}{1+\Delta\tau\,\lambda_k^2(\lambda_k^2-\mu)}$,
i.e. $\Lambda{f}(x)=\int_{\ssy D}G_{\ssy\Lambda}(x,y)f(y)\,dy$ for
$x\in{\overline D}$ and $f\in L^2(D)$.
Also, we set $G_{\ssy\Phi}(x,y):=-\partial_yG_{\ssy\Lambda}(x,y)=
-\sum_{k=1}^{\infty}\frac{\varepsilon_k(x)
\,\varepsilon_k'(y)}{1+\Delta\tau(\lambda_k^4-\mu\,\lambda_k^2)}$,
and define $\Phi:L^2(D)\to{\bfdot H}^4(D)$ by $\Phi
f(x):=\int_{\ssy D}G_{\ssy\Phi}(x,y)\,f(y)\;dy$ for $f\in L^2(D)$.
%
%
Also, for $m\in\nset$, we denote by $G_{\ssy \Lambda\Phi,m}$ the
Green function of the operator $\Lambda^{m-1}\Phi$.  In the sequel,  we will
use the symbol $C$ to denote a generic constant that is independent of $\Delta{t}$, $\Delta\tau$
and $\Delta{x}$, and may changes value from one line to the other. 
\par
Using \eqref{BackE2} and a simple induction argument, we conclude
that
$${\widehat U}^m=\sum_{j=1}^{\ssy m} \int_{\ssy\Delta_j}
\Lambda^{m-j}\Phi {\widehat W}(\tau,\cdot)\,d\tau,\quad
m=1,\dots,M,$$
which is written, equivalently, as follows:
\begin{equation}\label{Anaparastash1}
{\widehat U}^m(x)
=\int_0^{\tau_m}\!\!\!\int_{\ssy D}
\,{\widehat {\mathcal K}}_m(\tau;x,y)\,{\widehat W}(\tau,y)\,dyd\tau,
\quad\forall\,x\in{\overline D}, \ken m=1,\dots,M,
\end{equation}
where
${\widehat {\mathcal K}}_m(\tau;x,y):=\sum_{j=1}^m{\mathcal
X}_{\ssy\Delta_j}(\tau) \,G_{\ssy\Lambda\Phi,m-j+1}(x,y),
\quad\forall\,\tau\in[0,T],\ken\forall\,x,y\in D$.
\par
Let $m\in\{1,\dots,M\}$ and ${\mathcal E}^m:={\mathbb E}\big[
\|{\widehat U}^m -{\widehat u}(\tau_m,\cdot)\|_{\ssy 0,D}^2\big]$.
First, we use \eqref{Anaparastash1}, \eqref{HatUform},
\eqref{WNEQ2}, \eqref{Ito_Isom}, \eqref{HSxar} and
\eqref{L2L2_Bound}, to obtain
\begin{equation*}
\begin{split}
{\mathcal E}^m&={\mathbb E}\left[\,\int_{\ssy D}\left(
\int_0^{\ssy T} \!\!\!\int_{\ssy D} {\mathcal
X}_{(0,\tau_m)}(\tau) \,\big[{\widehat{\mathcal
K}}_m(\tau;x,y)-\Psi(\tau_m-\tau;x,y)\big]\,
{\widehat W}(\tau,y)\,dyd\tau\right)^2\,dx\right]\\
%
%
%
%
%
%
&\leq\int_0^{\tau_m}\left(\int_{\ssy D}\!\int_{\ssy D}
\,\big[{\widehat{\mathcal K}}_m(\tau;x,y)
-\Psi(\tau_m-\tau;x,y)\big]^2\,dydx\right)\,d\tau\\
%
%
&\leq\sum_{\ell=1}^{m}\int_{\ssy\Delta_{\ell}} \left(\int_{\ssy
D}\!\int_{\ssy D} \,\big[G_{\ssy\Lambda\Phi,m-\ell+1}(x,y)
-\Psi(\tau_m-\tau;x,y)\big]^2\,dydx\right)\,d\tau.\\
%
%
\end{split}
\end{equation*}
Now, we introduce the splitting
\begin{equation}\label{mainerrorF}
\sqrt{{\mathcal E}^m}\leq\,\sqrt{{\mathcal B}_1^m}+\sqrt{{\mathcal
B}_2^m},
\end{equation}
where
\begin{equation*}
\begin{split}
{\mathcal B}_1^m&:=\sum_{\ell=1}^{m}\int_{\ssy\Delta_{\ell}}
\left(\int_{\ssy D}\!\int_{\ssy D}
\,\big[G_{\ssy\Lambda\Phi,m-\ell+1}(x,y)
-\Psi(\tau_m-\tau_{\ell-1};x,y)\big]^2\,dydx\right)\,d\tau,\\
{\mathcal B}_2^m&:=\sum_{\ell=1}^{m}\int_{\ssy\Delta_{\ell}}
\left(\int_{\ssy D}\!\int_{\ssy D}
\,\big[\Psi(\tau_m-\tau_{\ell-1};x,y)
-\Psi(\tau_m-\tau;x,y)\big]^2\,dydx\right)\,d\tau.\\
\end{split}
\end{equation*}
\par
By the definition of the
Hilbert-Schmidt norm, we have
\begin{equation*}
\begin{split}
%
%
{\mathcal B}_1^m&\leq\Delta\tau\,\sum_{\ell=1}^m
\,\sum_{k=1}^{\infty} \int_{\ssy D}\left(\int_{\ssy D}
\,\big[G_{\ssy\Lambda\Phi,m-\ell+1}(x,y)\varphi_k(y)\;dy
-\int_{\ssy D}\Psi(\tau_m-\tau_{\ell-1};x,y)\varphi_k(y)\;dy
\right)^2\;dx\\
%
%
&\leq\sum_{k=1}^{\infty}\left(\, \sum_{\ell=1}^m\,\Delta\tau\,
\|\Lambda^{m-\ell}\Phi\varphi_k -{\mathcal
S}(\tau_m-\tau_{\ell-1})\varphi_k'\|^2_{\ssy 0,D}
\,\right)\\
&\leq\sum_{k=1}^{\infty}\left(\, \sum_{\ell=1}^m\,\Delta\tau\,
\|\Lambda^{m-\ell+1}\varphi_k' -{\mathcal
S}(\tau_m-\tau_{\ell-1})\varphi_k'\|^2_{\ssy 0,D}
\,\right)\\
%
%
%
%
&\leq\sum_{k=1}^{\infty}\,\lambda_k^2 \,\left(\, \sum_{\ell=1}^m
\,\Delta\tau\,\|\Lambda^{\ell}\varepsilon_k -{\mathcal
S}(\tau_{\ell})\varepsilon_k
\|^2_{\ssy 0,D}\,\right).\\
\end{split}
\end{equation*}
Let $\theta\in[0,\frac{1}{8})$. Using the deterministic error
estimate \eqref{Ydaspis900} and \eqref{SR_BOUND}, we obtain
\begin{equation}\label{trabajito}
\begin{split}
\sqrt{{\mathcal B}_1^m}&\leq\,C\,(\Delta{\tau})^{\theta}
\,\left(\,\sum_{k=1}^{\infty}\lambda_k^2\,\|\varepsilon_k\|^2_{\ssy
{\bfdot H}^{4\theta-2}}\,\right)^{\frac{1}{2}}\\
&\leq\,C\,(\Delta{\tau})^{\theta}\,\left(\,\sum_{k=1}^{\infty}
\tfrac{1}{\lambda_k^{1+8\,(\frac{1}{8}-\theta)}}\,\right)^{\frac{1}{2}}\\
&\leq\,C\,\tfrac{1}{\frac{1}{8}-\theta}\,(\Delta{\tau})^{\theta}.
\end{split}
\end{equation}
\par
 Using, again, the definition of the
Hilbert-Schmidt norm we have
\begin{equation}\label{Ydaspis950}
\begin{split}
{\mathcal B}_2^m=&\,\sum_{k=1}^{\infty}
\sum_{\ell=1}^m\int_{\ssy\Delta_{\ell}} \|{\mathcal
S}(\tau_m-\tau_{\ell-1})\varphi_k' -{\mathcal
S}(\tau_m-\tau)\varphi_k'\|^2_{\ssy 0,D} \,d\tau\\
=&\,\sum_{k=1}^{\infty}\lambda_k^2
\sum_{\ell=1}^m\int_{\ssy\Delta_{\ell}} \|{\mathcal
S}(\tau_m-\tau_{\ell-1})\varepsilon_k -{\mathcal
S}(\tau_m-\tau)\varepsilon_k\|^2_{\ssy 0,D} \,d\tau
\end{split}
\end{equation}
Observing that
${\mathcal
S}(t)\varepsilon_k=e^{-\lambda_k^2(\lambda_k^2-\mu)t}\,\varepsilon_{k}$
for $t\ge 0$, \eqref{Ydaspis950} yields
\begin{equation}\label{pipinis1}
\begin{split}
{\mathcal B}_2^m=&\,\sum_{k=1}^{\infty}\lambda_k^2\,
\sum_{\ell=1}^m\int_{\ssy\Delta_{\ell}} \left(\int_{\ssy D}\left[
e^{-(\lambda_k^4-\mu\,\lambda_k^2)(\tau_m-\tau_{\ell-1})}
-e^{-(\lambda_k^4-\mu\,\lambda_k^2)(\tau_m-\tau)}\right]^2
\varepsilon_k^2(x)\,dx\right)
\,d\tau\\
=&\,\sum_{k=1}^{\infty}\lambda_k^2\,
\sum_{\ell=1}^m\int_{\ssy\Delta_{\ell}}
e^{-2(\lambda_k^4-\mu\,\lambda_k^2)(\tau_m-\tau)} \left[1-
e^{-(\lambda_k^4-\mu\,\lambda_k^2)(\tau-\tau_{\ell-1})}\right]^2\,d\tau\\
\leq&\,{\mathcal B}_{2,1}^{m}+{\mathcal B}_{2,2}^m,
\end{split}
\end{equation}
where
\begin{equation*}
\begin{split}
{\mathcal B}_{2,1}^m:=&\,
\sum_{k=1}^{\kappa}\lambda_k^2\,
\sum_{\ell=1}^m\int_{\ssy\Delta_{\ell}}
e^{-2\lambda_k^2(\lambda_k^2-\mu)(\tau_m-\tau)} \left[1-
e^{-(\lambda_k^4-\mu\,\lambda_k^2)(\tau-\tau_{\ell-1})}\right]^2\,d\tau,\\
{\mathcal B}_{2,2}^m:=&\,
\sum_{k=\kappa+1}^{\infty}\lambda_k^2\,
\sum_{\ell=1}^m\int_{\ssy\Delta_{\ell}}
e^{-2\lambda_k^2(\lambda_k^2-\mu)(\tau_m-\tau)} \left[1-
e^{-(\lambda_k^4-\mu\,\lambda_k^2)(\tau-\tau_{\ell-1})}\right]^2\,d\tau.\\
\end{split}
\end{equation*}
First, we estimate ${\mathcal B}_{2,1}^m$ and ${\mathcal B}_{2,2}^m$ as follows
\begin{equation}\label{pipinis2}
\begin{split}
{\mathcal B}_{2,2}^m\leq&\,\sum_{k=\kappa+1}^{\infty}\lambda_k^2\,
\big(1-e^{-\lambda_k^2(\lambda_k^2-\mu)\,\Delta\tau }\,\big)^2
\left[\, \int_0^{\tau_m}
e^{-2(\lambda_k^4-\mu\,\lambda_k^2)(\tau_m-\tau)}
\,d\tau\,\right]\\
\leq&\,\tfrac{1}{2}\,\sum_{k=\kappa+1}^{\infty}
\tfrac{1-e^{-2\lambda_k^2(\lambda_k^2-\mu)
\,\Delta\tau}}{\lambda_k^2-\mu}\\
\leq&\,\tfrac{(\kappa+1)^2}{2(1+2\kappa)}\,\sum_{k=\kappa+1}^{\infty}
\tfrac{1-e^{-2\lambda_k^2(\lambda_k^2-\mu)
\,\Delta\tau}}{\lambda_k^2}\\
\leq&\,C\,\sum_{k=1}^{\infty}
\tfrac{1-e^{-c_0\,\lambda_k^4\,\Delta\tau}}{\lambda_k^2}
\end{split}
\end{equation}
with $c_0=\frac{2(1+2\kappa)}{(\kappa+1)^2}$, and
\begin{equation}\label{pipinis3}
\begin{split}
{\mathcal B}_{2,1}^m\leq&\,C\,\sum_{k=1}^{\kappa}
\sum_{\ell=1}^m\int_{\ssy\Delta_{\ell}} \left[1-
e^{-(\lambda_k^4-\mu\,\lambda_k^2)(\tau-\tau_{\ell-1})}\right]^2\,d\tau\\
\leq&\,C\,\sum_{k=1}^{\kappa}
\sum_{\ell=1}^m\int_{\ssy\Delta_{\ell}} \left[
(\lambda_k^4-\mu\,\lambda_k^2)(\tau-\tau_{\ell-1})\right]^2\,d\tau\\
\leq&\,C\,(\Delta\tau)^2.\\
\end{split}
\end{equation}
Finally, we  combine \eqref{pipinis1}, \eqref{pipinis2}, \eqref{pipinis3} and
\eqref{corv10}, to obtain
\begin{equation}\label{Ydaspis952}
\sqrt{{\mathcal
B}_2^m}\leq\,C\,\,\left(1+(\Delta{\tau})^{\frac{3}{4}}
+(\Delta{\tau})^{\frac{7}{4}}\right)^{\frac{1}{2}}
\,\,(\Delta\tau)^{\frac{1}{8}}.
\end{equation}
\par
The estimate \eqref{ElPasso} follows by \eqref{mainerrorF},
\eqref{trabajito} and \eqref{Ydaspis952}. $\mybox$
\end{proof}
%
%
%
%
%
%
\section{Convergence of the Fully-Discrete Approximations}\label{SECTION44}
To get an error estimate for the fully-discrete approximations of
${\widehat u}$ defined by \eqref{FullDE1}--\eqref{FullDE2}, we
proceed by comparing them with their time-discrete approximations
defined by \eqref{BackE1}--\eqref{BackE2} and using a discrete Duhamel
principle (cf. \cite{KZ2010}, \cite{YubinY05}).
\subsection{The Deterministic Case}
The Backward Euler finite element approximations
of the solution to \eqref{Det_Parab} are defined as follows:
first, set
\begin{equation}\label{DetFD1}
W_h^0:=P_hw_0,
\end{equation}
and then, for $m=1,\dots,M$, find $W_h^m\in M_h^r$ such that
\begin{equation}\label{DetFD2}
W_h^m-W_h^{m-1}+\Delta\tau\,\Lambda_{\ssy B,h}W_h^m=0,
\end{equation}
which is possible when $\mu^2\,\Delta\tau<4$.
%
\par
Next, we derive a discrete in time $L^2_t(L_x^2)$ estimate for the
error approximating the Backward Euler time-discrete approximations of the
solution to \eqref{Det_Parab} defined in \eqref{BEDet1}-\eqref{BEDet2},
by the Backward Euler finite element approximations defined in
\eqref{DetFD1}-\eqref{DetFD2}.
The main difference with the case $\mu=0$ which has been considered
in \cite{KZ2010}, is that, our assumption \eqref{mu_condition} on $\mu$,
can not ensure the coerciveness  of the discrete elliptic operator
$\Lambda_{\ssy B,h}$.
%
%
\begin{theorem}\label{Aygo_Kokora}
Let $r=2$ or $3$, $w$ be the solution to the problem
\eqref{Det_Parab}, $(W^m)_{m=0}^{\ssy M}$ be the time-discrete
approxi\-mations of $w$ defined in \eqref{BEDet1}-\eqref{BEDet2},
and $(W_h^m)_{m=0}^{\ssy M}\subset M_h^r$ be the fully-discrete
approximations of $w$ defined in \eqref{DetFD1}-\eqref{DetFD2}.
Also, we assume that $\mu^2\,\Delta\tau<\frac{1}{4}$.
If $w_0\in {\bfdot H}^2(D)$, then, there exists a nonnegative
constant ${\widehat c}_1$, independent of $h$ and
$\Delta{\tau}$, such that
\begin{equation}\label{tiger_river1}
\left(\,\sum_{m=1}^{\ssy M}\Delta{\tau}\,\|W^m-W_h^m\|^2_{\ssy
0,D} \,\right)^{\frac{1}{2}}\leq \,{\widehat
c}_1\,\,h^{\ell_{\star}(r)\,\theta} \,\,\|w_0\|_{\ssy {\bfdot
H}^{\xi_{\star}(r,\theta)}},\quad\forall\,\theta\in[0,1],
\end{equation}
where
\begin{equation}\label{Basilico_4}
\ell_{\star}(r):=\left\{ \aligned
&2\quad\text{\rm if}\ken r=2\\
&4\quad\text{\rm if}\ken r=3\\
\endaligned
\right.
\quad\text{and}\quad
\xi_{\star}(r,\theta):=(r+1)\,\theta-2.
\end{equation}
\end{theorem}
%
%
%
%
%
%
%
\begin{proof}
The error estimate \eqref{tiger_river1} follows by interpolation,
after showing that holds for $\theta=0$ and $\theta=1$.  In the sequel, 
we will use the symbol $C$ to denote a generic constant that is
independent of $\Delta\tau$ and $h$, and may changes value from one line to the other. 
%
%
%
%
%
%
\par
Let $E^m:=W_h^m-W^m$ for $m=0,\dots,M$. First, use \eqref{DetFD2} and
\eqref{BEDet2} to obtain
\begin{gather}
W_h^m-W_h^{m-1}+\Delta{\tau}\,{\widetilde\Lambda}_{\ssy B,h}W^m_h
=\Delta{\tau}\,\mu^2\,W_h^m,\label{Darm_1a}\\
W^m-W^{m-1}+\Delta{\tau}\,{\widetilde\Lambda}_{\ssy B}W^m
=\Delta{\tau}\,\mu^2\,W^m\label{Darm_1b}
\end{gather}
for $m=1,\dots,M$. Then, combine \eqref{Darm_1a} and
\eqref{Darm_1b}, to get the following error equation
\begin{equation}\label{error_equation_tilde}
\begin{split}
{\widetilde T}_{\ssy B,h}(E^m-E^{m-1})+\Delta\tau\,E^m
=&\,\Delta\tau\,\mu^2\,{\widetilde T}_{\ssy B,h}E^m
-\Delta{\tau}\,({\widetilde T}_{\ssy B}-{\widetilde T}_{\ssy
B,h}){\widetilde\Lambda}_{\ssy B}W^m,
\quad m=1,\dots,M.\\
\end{split}
\end{equation}
Taking the $L^2(D)-$inner product with $E^m$ of both sides of
\eqref{error_equation_tilde}, it follows that
\begin{equation*}
\begin{split}
{\widetilde\gamma}_{\ssy B,h}(E^m-E^{m-1},E^m)
+\Delta\tau\,\|E^m\|_{\ssy 0,D}^2
=&\,\Delta\tau\,\mu^2\,{\widetilde\gamma}_{\ssy B,h}(E^m,E^m)\\
&-\Delta{\tau}\,(({\widetilde T}_{\ssy B}
-{\widetilde T}_{\ssy B,h}){\widetilde\Lambda}_{\ssy B}W^m,E^m)_{\ssy 0,D},
\quad m=1,\dots,M,
\end{split}
\end{equation*}
from which, after using \eqref{innerproduct}, we conclude that
\begin{equation}\label{Parlamento2}
\begin{split}
{\widetilde\gamma}_{\ssy B,h}(E^m,E^m)
+\Delta\tau\,\|E^m\|_{\ssy 0,D}^2
\leq&\,{\widetilde\gamma}_{\ssy B,h}(E^{m-1},E^{m-1})
+2\,\Delta\tau\,\mu^2\,{\widetilde\gamma}_{\ssy B,h}(E^m,E^m)\\
&+\Delta\tau\,\|({\widetilde T}_{\ssy B}-{\widetilde T}_{\ssy B,h})
{\widetilde\Lambda}_{\ssy B}W^m\|_{\ssy 0,D}^2,
\quad m=1,\dots,M.\\
\end{split}
\end{equation}
Since $2\Delta\tau\mu^2<1$, \eqref{Parlamento2} yields
\begin{equation}\label{Parlamento3}
{\widetilde\gamma}_{\ssy B,h}(E^m,E^m)
\leq\,\tfrac{1}{1-2\,\Delta\tau\,\mu^2}\,\left[\,
{\widetilde\gamma}_{\ssy B,h}(E^{m-1},E^{m-1})+\Delta\tau\,\|({\widetilde T}_{\ssy B}
-{\widetilde T}_{\ssy B,h}){\widetilde\Lambda}_{\ssy B}W^m\|_{\ssy 0,D}^2\right]
\end{equation}
for $m=1,\dots,M$. Applying a simple induction argument based on \eqref{Parlamento2}
and then using that $4\Delta\tau\mu^2<1$,  we get
\begin{equation}\label{Parlamento4}
\max_{0\leq{m}\leq{\ssy M}}{\widetilde\gamma}_{\ssy B,h}(E^m,E^m)
\leq\,C\,\left[\,{\widetilde\gamma}_{\ssy B,h}(E^0,E^0)
+\Delta\tau\sum_{\ell=1}^{\ssy M}
\|({\widetilde T}_{\ssy B}-{\widetilde T}_{\ssy B,h})
{\widetilde\Lambda}_{\ssy B}W^{\ell}\|_{\ssy 0,D}^2\,\right].
\end{equation}
Summing with respect to $m$ from $1$ up to $M$, using \eqref{Parlamento4}
and observing that ${\widetilde T}_{\ssy B,h}E^0=0$,  \eqref{Parlamento2} 
gives
\begin{equation}\label{citah3}
\sum_{m=1}^{\ssy M}\Delta\tau\,\|E^m\|_{\ssy 0,D}^2\leq\,C
\,\sum_{m=1}^{\ssy M}\Delta{\tau}\,\|({\widetilde T}_{\ssy B}
-{\widetilde T}_{\ssy B,h}){\widetilde\Lambda}_{\ssy B}W^m\|_{\ssy 0,D}^2.
\end{equation}
\par
Let $r=3$. Then, by \eqref{ARA2}, \eqref{citah3} and the
Poincar{\'e}-Friedrich inequality, we obtain
\begin{equation}\label{citah_karim_4}
\begin{split}
\left(\,\sum_{m=1}^{\ssy M}\Delta\tau\,\|E^m\|_{\ssy
0,D}^2\,\right)^{\frac{1}{2}}\leq&\,C\,h^4\,\left(
\,\sum_{m=1}^{\ssy M}\Delta\tau\, \|{\widetilde\Lambda}_{\ssy
B}W^m\|_{\ssy 0,D}^2\right)^{\frac{1}{2}}\\
\leq&\,C\,h^4\,\left[\,\sum_{m=1}^{\ssy M}\Delta\tau\,\left(\,
\|\partial_x^4W^m\|_{\ssy 0,D}^2 +\|\partial_x^2W^m\|_{\ssy 0,D}^2
+\|\partial_x^1W^m\|_{\ssy 0,D}^2\,\right)\,\right]^{\frac{1}{2}}.\\
\end{split}
\end{equation}
Taking the $L^2(D)-$inner product of \eqref{BEDet2} with
$\partial^4W^m$ and then integrating by parts, we obtain
\begin{equation}\label{citah_ksa_5}
(\partial^2W^m-\partial^2W^{m-1},\partial^2W^m)_{\ssy 0,D}
+\Delta\tau\,\|\partial^4W^m\|_{\ssy
0,D}^2+\mu\,\Delta\tau\,(\partial^2W^m,\partial^4 W^m)_{\ssy
0,D}=0,\quad m=1,\dots,M.
\end{equation}
Using \eqref{innerproduct}, \eqref{citah_ksa_5} and
the Cauchy-Schwarz inequality we obtain
\begin{equation*}
\|\partial^2W^m\|_{\ssy
0,D}^2+2\,\Delta\tau\,\|\partial^4W^m\|_{\ssy 0,D}^2
\leq\,\|\partial^2W^{m-1}\|_{\ssy 0,D}^2
+2\,\mu\,\Delta\tau\,\|\partial^2W^{m-1}\|_{\ssy 0,D}
\,\|\partial^4W^m\|_{\ssy 0,D},\quad
m=1,\dots,M,
\end{equation*}
which, after using the geometric mean inequality, yields
\begin{equation}\label{citah_ksa_6}
\|\partial^2W^m\|_{\ssy
0,D}^2+\Delta\tau\,\|\partial^4W^m\|_{\ssy 0,D}^2
\leq\|\partial^2W^{m-1}\|_{\ssy 0,D}^2
+\Delta\tau\,\mu^2\,\|\partial^2W^m\|_{\ssy 0,D}^2,\quad
m=1,\dots,M.
\end{equation}
Since  $2\,\mu^2\,\Delta\tau<1$, from
\eqref{citah_ksa_6} follows that
\begin{equation*}
\begin{split} \|\partial^2W^m\|_{\ssy
0,D}^2\leq&\,\tfrac{1}{1-\mu^2\,\Delta\tau}
\,\|\partial^2W^{m-1}\|^2_{\ssy 0,D}\\
\leq&\,e^{2\mu^2\Delta\tau}
\,\|\partial^2W^{m-1}\|_{\ssy 0,D}^2,
\quad m=1,\dots,M,\\
\end{split}
\end{equation*}
from which, applying a simple induction argument, we conclude that
\begin{equation}\label{citah_ksa_7}
\max_{0\leq{m}\leq{\ssy M}}\|\partial^2W^m\|_{\ssy 0,D}^2
\leq\,C\,\|w_0\|_{\ssy 2,D}^2.
\end{equation}
Next, sum both side of \eqref{citah_ksa_6} with respect to $m$,
from $1$ up to $M$, and use \eqref{citah_ksa_7} to conclude that
\begin{equation}\label{citah_ksa_8}
\sum_{m=1}^{\ssy M}\Delta\tau\,\|\partial^4W^m\|_{\ssy 0,D}^2
\leq\,C\,\|w_0\|_{\ssy 2,D}^2.
\end{equation}
Taking the $L^2(D)-$inner product of \eqref{BEDet2} with
$\partial^2W^m$, and then integrating by parts, it follows that
\begin{equation}\label{citah_karim_9}
(\partial W^m-\partial W^{m-1},\partial W^m)_{\ssy 0,D}
+\Delta\tau\,\|\partial^3W^m\|_{\ssy
0,D}^2+\mu\,\Delta\tau\,(\partial W^m,\partial^3W^m)_{\ssy
0,D}=0,\quad m=1,\dots,M.
\end{equation}
Using \eqref{innerproduct},
\eqref{citah_karim_9}, the Cauchy-Schwarz inequality and
the geometric mean inequality, we obtain
\begin{equation*}\label{citah_karim_10}
\|\partial W^m\|_{\ssy
0,D}^2+\Delta\tau\,\|\partial^3W^m\|_{\ssy 0,D}^2
\leq\|\partial W^{m-1}\|_{\ssy 0,D}^2
+\Delta\tau\,\mu^2\,\|\partial W^m\|_{\ssy 0,D}^2,\quad
m=1,\dots,M.
\end{equation*}
Since $2\,\mu^2\,\Delta\tau< 1$,
proceeding as in obtaining \eqref{citah_ksa_7} and
\eqref{citah_ksa_8} from \eqref{citah_ksa_6}, we arrive at
\begin{equation}\label{citah_citah_10}
\max_{0\leq{m}\leq{\ssy M}}\|\partial W^m\|_{\ssy
0,D}^2+\sum_{m=1}^{\ssy M}\Delta\tau\,\|\partial^3W^m\|_{\ssy
0,D}^2 \leq\,C\,\|w_0\|_{\ssy 1,D}^2.
\end{equation}
Thus, combining \eqref{citah_karim_4}, 
\eqref{citah_ksa_8}, \eqref{citah_ksa_7},
\eqref{citah_citah_10}  and \eqref{H_equiv}, we obtain
\begin{equation}\label{final_karim_1}
\left(\,\sum_{m=1}^{\ssy M}\Delta\tau\,\|E^m\|_{\ssy
0,D}^2\,\right)^{\frac{1}{2}}\leq\,C\,h^4\,\|w_0\|_{\ssy {\bfdot
H}^2}.
\end{equation}
\par
Let $r=2$. Then, by \eqref{ARA2}, \eqref{citah3} and the
Poincar{\'e}-Friedrich inequality, we obtain
\begin{equation}\label{citah4}
\begin{split}
\left(\,\sum_{m=1}^{\ssy M}\Delta\tau\,\|{E}^m\|_{\ssy
0,D}^2\,\right)^{\frac{1}{2}}\leq&\,C\,h^2\,\left(
\,\sum_{m=1}^{\ssy M}\Delta\tau\, \|{\widetilde\Lambda}_{\ssy
B}W^m\|_{\ssy -1,D}^2\right)^{\frac{1}{2}}\\
\leq&\,C\,h^2\,\left[\,\sum_{m=1}^{\ssy M}\Delta\tau\,\left(\,
\|\partial^3W^m\|_{\ssy 0,D}^2
+\|\partial W^m\|_{\ssy 0,D}^2\,\right)\,\right]^{\frac{1}{2}}.\\
\end{split}
\end{equation}
Combining, now, \eqref{citah4}, \eqref{citah_citah_10} and
\eqref{H_equiv}, we obtain
\begin{equation}\label{final_karim_2}
\left(\,\sum_{m=1}^{\ssy M}\Delta\tau\,\|E^m\|_{\ssy
0,D}^2\,\right)^{\frac{1}{2}}
\leq\,C\,h^2\,\|w_0\|_{\ssy {\bfdot H}^1}.
\end{equation}
\par
Thus, relations \eqref{final_karim_1} and
\eqref{final_karim_2} yield \eqref{tiger_river1} and
\eqref{Basilico_4} for $\theta=1$.
\par
Since $\mu^2\,\Delta\tau<1$, using \eqref{Darm_1a}, we have
\begin{equation*}
{\widetilde T}_{{\ssy B},h}(W_h^m-W_h^{m-1})
+\Delta\tau\,W_h^m=\Delta\tau\,\mu^2\,{\widetilde T}_{\ssy B,h}W_h^m,
\quad m=1,\dots,M,
\end{equation*}
from which, after taking the $L^2(D)-$inner product with
$W_h^m$, we obtain
\begin{equation}\label{Harvard_1}
{\widetilde\gamma}_{\ssy B,h}(W_h^m-W_h^{m-1},W_h^m)_{\ssy
0,D}+\Delta\tau\,\|W^m_h\|_{\ssy 0,D}^2=\Delta\tau\,\mu^2\,
{\widetilde\gamma}_{\ssy B,h}(W_h^m,W_h^m),\quad m=1,\dots,M.
\end{equation}
Then we combine \eqref{Harvard_1} with \eqref{innerproduct} to have 
\begin{equation}\label{Yale_1}
(1-2\,\Delta\tau\,\mu^2)\,{\widetilde\gamma}_{\ssy B,h}(W_h^m,W_h^m)
+2\,\Delta\tau\,\|W_h^m\|_{\ssy 0,D}^2\leq\,{\widetilde\gamma}_{\ssy
B,h}(W_h^{m-1},W_h^{m-1}),\quad m=1,\dots,M.
\end{equation}
Since  $4\,\mu^2\,\Delta\tau<1$, \eqref{Yale_1} yields that
\begin{equation*}
\begin{split}
{\widetilde\gamma}_{\ssy B,h}(W_h^m,W_h^m)
\leq&\,\tfrac{1}{1-2\,\mu^2\,\Delta\tau}
\,{\widetilde\gamma}_{\ssy B,h}(W_h^{m-1},W_h^{m-1})\\
\leq&\,e^{4\mu^2\Delta\tau}
\,{\widetilde\gamma}_{\ssy B,h}(W_h^{m-1},W_h^{m-1}),
\quad m=1,\dots,M,\\
\end{split}
\end{equation*}
from which, applying a simple induction argument, we conclude that
\begin{equation}\label{Yale_2}
\max_{0\leq{m}\leq{\ssy M}}{\widetilde\gamma}_{\ssy B,h}(W_h^{m},W_h^{m})
\leq\,C\,{\widetilde\gamma}_{\ssy B,h}(W_h^0,W_h^0).
\end{equation}
Summing with respect to $m$ from 1 up to $M$, and using
\eqref{Yale_2}, \eqref{Yale_1} gives
\begin{equation}\label{Harvard_2}
\Delta\tau\,\sum_{m=1}^{\ssy
M}\|W_h^m\|_{\ssy 0,D}^2\leq\,C\,{\widetilde\gamma}_{\ssy
B,h}(W_h^0,W_h^0)_{\ssy 0,D}.
\end{equation}
Finally, using \eqref{Harvard_2}, \eqref{PanwFragma1} and \eqref{minus_equiv} we
obtain
\begin{equation}\label{Harvard_3}
\begin{split}
\sum_{m=1}^{\ssy M}\Delta\tau\,\|W_h^m\|_{\ssy
0,D}^2
\leq&\,C\,({\widetilde T}_{\ssy B,h}w_0,w_0)_{\ssy 0,D}\\
\leq&\,C\,\|w_0\|^2_{\ssy -2,D}\\
\leq&\,C\,\|w_0\|^2_{\ssy {\bfdot H}^{-2}}.\\
\end{split}
\end{equation}
Finally, combine \eqref{Harvard_3} with \eqref{Ydaspis904} to get
$$\left(\,\sum_{m=1}^{\ssy M}\Delta\tau\,\|W^m-W_h^m\|_{\ssy
0,D}^2\,\right)^{\frac{1}{2}}\leq\,C\,\|w_0\|_{\ssy {\bfdot
H}^{-2}},$$
which yields \eqref{tiger_river1} and \eqref{Basilico_4} for
$\theta=0$.
\end{proof}
%
%
%
%
%
\subsection{The Stochastic Case}
Our first step is to show the existence of a Green function for
the solution operator of a discrete elliptic problem.
%
%
\begin{lemma}\label{prasinolhmma}
Let $r=2$ or $3$, $\epsilon>0$ with $\mu^2\epsilon<4$, $f\in
L^2(D)$ and $\psi_h\in M_h^r$ such that
\begin{equation}\label{bohqos_pro}
\psi_h+\epsilon\,\Lambda_{\ssy B,h}\psi_h=P_hf.
\end{equation}
Then there exists a function $A_{\epsilon,h}\in H^2(D\times D)$
such that $A_{\epsilon,h}\left|_{\partial(D\times D)}\right.=0$
and
\begin{equation}\label{prasinogreen}
\psi_h(x)=\int_{\ssy D} A_{h,\epsilon}(x,y)\,f(y)\,dy
\quad\forall\,x\in{\overline D}
\end{equation}
and $A_{h,\epsilon}(x,y)=A_{h,\epsilon}(y,x)$ for $x,y\in
{\overline D}.$
\end{lemma}
%
%
%
%
%
%
%
%
\begin{proof}
Let $\delta_{\epsilon,h}:M_h^r\times M_h^r\rightarrow\rset$ be the
inner product on $M_h^r$ given by
\begin{equation*}
\begin{split}
\delta_{\epsilon,h}(\phi,\chi):=&\,\epsilon\,(\Lambda_{\ssy
B,h}\phi,\chi)_{\ssy 0,D}+(\phi,\chi)_{\ssy
0,D}\\
=&\,\epsilon\,(\phi'',\chi'')_{\ssy 0,D}
+\epsilon\,\mu\,(\phi'',\chi)_{\ssy 0,D}
+(\phi,\chi)_{\ssy0,D},\quad\forall\,\phi,\chi\in M^r_h.\\
\end{split}
\end{equation*}
We can construct a basis $(\chi_j)_{j=1}^{n_h}$ of $M_h^r$ which
is $L^2(D)-$orthonormal, i.e., $(\chi_i,\chi_j)_{\ssy
0,D}=\delta_{ij}$ for $i,j=1,\dots,n_h$, and
$\delta_{\epsilon,h}-$orthogonal, i.e., there exist
$(\lambda_{\epsilon,h,\ell})_{\ell=1}^{\ssy
n_h}\subset(0,+\infty)$ such that
$\delta_{\epsilon,h}(\chi_i,\chi_j)
=\lambda_{\epsilon,h,i}\,\delta_{ij}$ for $i,j=1,\dots,n_h$ (see
Section 8.7 in \cite{Golub}).
Thus, there are $(\mu_j)_{j=1}^{\ssy n_h}\subset\rset$ such that
$\psi_h=\sum_{j=1}^{\ssy n_h}\mu_j\,\chi_j$, and
\eqref{bohqos_pro} is equivalent to $\mu_i
=\frac{1}{\lambda_{\epsilon,h,i}}\,(f,\chi_i)_{\ssy 0,D}$ for
$i=1,\dots,n_h$. Finally, we obtain \eqref{prasinogreen} with
$A_{h,\epsilon}(x,y)=\sum_{j=1}^{\ssy n_h}
\frac{\chi_j(x)\chi_j(y)}{\lambda_{\epsilon,h,j}}$.
\end{proof}
%
%
%
%
%
%
Our second step is to compare, in a discrete in time
$L^{\infty}_t(L^2_{\ssy P}(L^2_x))$ norm, the Backward Euler
time-discrete approximations of ${\widehat u}$ with the Backward
Euler finite element approximations of ${\widehat u}$.
%
%
%
%
%
%
\begin{proposition}\label{Tigrakis}
Let $r=2$ or $3$, ${\widehat u}$ be the solution of the problem
\eqref{AC2}, $({\widehat U}_h^m)_{m=0}^{\ssy M}$ be the Backward
Euler finite element approximations of ${\widehat u}$ defined in
\eqref{FullDE1}-\eqref{FullDE2}, and $({\widehat U}^m)_{m=0}^{\ssy
M}$ be the Backward Euler time-discrete approximations of
${\widehat u}$ defined in \eqref{BackE1}-\eqref{BackE2}.
Also, we assume that  $\mu^2\,\Delta\tau\leq \frac{1}{4}$.
Then, there exists a nonnegative constant ${\widehat c}_2$,
independent of $\Delta{x}$, $\Delta{t}$, $h$ and $\Delta\tau$,
such that
\begin{equation}\label{Lasso1}
\max_{0\leq{m}\leq {\ssy M}}\left({\mathbb E}\left[
\big\|{\widehat U}_h^m -{\widehat U}^m\big\|^2_{\ssy
0,D}\right]\right)^{\half} \leq\,{\widehat
c}_2\,\epsilon^{-\frac{1}{2}} \,\,\,h^{\nu(r)-\epsilon},
\quad\forall\,\epsilon\in(0,\nu(r)],
\end{equation}
where
\begin{equation}\label{Nice_Day_1}
\nu(r):=\left\{ \aligned
&\tfrac{1}{3}\quad\text{\rm if}\ken r=2\\
&\tfrac{1}{2}\hskip0.35truecm\text{\rm if}\ken r=3\\
\endaligned
\right..
\end{equation}
\end{proposition}
%
%
%
%
%
%
%
%
%
\begin{proof}
Let $I:L^2(D)\to L^2(D)$ be the identity operator and
$\Lambda_h:L^2(D)\to M^r_h$ be the inverse discrete elliptic
operator given by $\Lambda_h:=(I+\Delta\tau\,\Lambda_{\ssy
B,h})^{-1}P_h$, having a Green function $G_{\ssy
\Lambda_h}=A_{h,\Delta\tau}$ according to Lemma~\ref{prasinolhmma}
and taking into account that $\mu^2\,\Delta\tau<4$.
Also, we define an operator $\Phi_h:L^2(D)\rightarrow M_h^r$ by
$(\Phi_h f)(x):=\int_{\ssy D}G_{\ssy\Phi_h}(x,y)\,f(y)\;dy$ for
$f\in L^2(D)$ and $x\in{\overline D}$, where $G_{\ssy
\Phi_h}(x,y)=-\partial_yG_{\ssy\Lambda_h}(x,y)$. Then, we have
that $\Lambda_hf'=\Phi_hf$ for all $f\in H^1(D)$.
Also, for $\ell\in\nset$, we denote by
$G_{\ssy\Lambda_h,\Phi_h,\ell}$ the Green function of
$\Lambda_h^{\ell}\Phi_h$.  In the sequel,  we will use the symbol $C$ to denote a generic
constant that is independent of $\Delta{t}$, $\Delta{x}$, $h$ and $\Delta\tau$,
and may changes value from one line to the other. 
\par
Applying, an induction argument, from \eqref{FullDE2} we
conclude that
$${\widehat U}_h^m=\sum_{j=1}^{\ssy m} \int_{\ssy\Delta_j}
\Lambda_h^{m-j}\Phi_h{\widehat W}(\tau,\cdot)\,d\tau,\quad
m=1,\dots,M,$$
which is written, equivalently, as follows:
\begin{equation}\label{Anaparastash2}
{\widehat U}_h^m(x)=\int_0^{\tau_m}\!\!\!\int_{\ssy D}
\,{\widehat{\mathcal D}}_{h,m}(\tau;x,y)\,{\widehat W}(\tau,y)
\,dyd\tau\quad\forall\,x\in{\overline D}, \ken m=1,\dots,M,
\end{equation}
where
${\widehat{\mathcal D}}_{h,m}(\tau;x,y) :=\sum_{j=1}^m{\mathcal
X}_{\ssy\Delta_j}(\tau)
\,G_{\ssy\Lambda_h,\Phi_h,m-j}(x,y)\quad\forall
\,\tau\in[0,T],\ken\forall\,x,y\in D$.
Using \eqref{Anaparastash1}, \eqref{Anaparastash2}, the
It{\^o}-isometry property of the stochastic integral,
\eqref{HSxar} and the Cauchy-Schwarz inequality, we get
\begin{equation*}
\begin{split}
{\mathbb E}\left[
\|{\widehat U}^m-{\widehat U}_h^m\|_{\ssy 0,D}^2
\right]
&\leq\int_0^{\tau_m}\Big(\int_{\ssy D}\!\int_{\ssy D}
\,\big[{\widehat{\mathcal K}}_m(\tau;x,y)
-{\widehat{\mathcal D}}_{h,m}(\tau;x,y)\big]^2
\,dydx\Big)\,d\tau\\
&\leq\, \sum_{j=1}^m\int_{\ssy\Delta_{j}}\,\|\Lambda^{m-j}\Phi
-\Lambda_h^{m-j}\Phi_h\|_{\ssy\rm HS}^2\,d\tau,
\quad m=1,\dots,M,\\
\end{split}
\end{equation*}
where $\Lambda$ and $\Phi$ are the operators defined in the proof
of Theorem~\ref{TimeDiscreteErr1}. Now, we use the definition of
the Hilbert-Schmidt norm and the deterministic error estimate
\eqref{tiger_river1}, to obtain
\begin{equation*}
\begin{split}
%
%
{\mathbb E}\left[ \|{\widehat U}^m-{\widehat U}_h^m\|_{\ssy 0,D}^2
\right] &\leq\,\sum_{j=1}^m\,\Delta\tau
\left[\,\sum_{k=1}^{\infty} \|\Lambda^{m-j}\Phi\varphi_k
-\Lambda^{m-j}_h\Phi_h\varphi_k\|^2_{\ssy 0,D}
\,\right]\\
%
%
&\leq\,\sum_{k=1}^{\infty}\left[\, \sum_{\ell=1}^m\,\Delta\tau\,
\|\Lambda^{\ell}\varphi_k' -\Lambda_h^{\ell}\varphi_k'\|^2_{\ssy
0,D}
\,\right]\\
%
%
&\leq\,\sum_{k=1}^{\infty}\lambda_k^2\,\left[\,
\sum_{\ell=1}^m\,\Delta\tau\, \|\Lambda^{\ell}\varepsilon_k
-\Lambda_h^{\ell}\varepsilon_k\|^2_{\ssy 0,D}
\,\right]\\
%
&\leq\,C\,h^{2\,\ell_{\star}(r)\,\theta}
\,\sum_{k=1}^{\infty}\lambda_k^2\,\|\varepsilon_k\|^2_{\ssy
{\bfdot H}^{\xi_{\star}(r,\theta)}},\quad m=1,\dots,M,
\quad\forall\,\theta\in[0,1].\\
\end{split}
\end{equation*}
Thus, we arrive at
\begin{equation}\label{Easter2012a}
\max_{1\leq{m}\leq{\ssy M}}\left(\,{\mathbb E}\left[ \|{\widehat
U}^m-{\widehat U}_h^m\|_{\ssy 0,D}^2
\right]\,\right)^{\frac{1}{2}}
\leq\,C\,h^{\ell_{\star}(r)\,\theta}\,\left(\,\sum_{k=1}^{\infty}
\lambda_k^{-\left[1+\frac{2\,(r+1)}{\ell_{\star}(r)}\left(\nu(r)
-\ell_{\star}(r)\,\theta\right)\right]}\,\right)^{\frac{1}{2}},
\quad\forall\,\theta\in[0,1].
\end{equation}
It is easily seen that the series in the right hand side of \eqref{Easter2012a}
convergences iff $\nu(r)>\ell_{\star}(r)\,\theta$. Thus, setting
$\epsilon=\nu(r)-\ell_{\star}(r)\,\theta$, requiring $\epsilon\in(0,\nu(r)]$,
and combining \eqref{Easter2012a} and \eqref{SR_BOUND}, we arrive at 
the estimate \eqref{Lasso1}.
\end{proof}
%
%
%
%
%
%
%
\par
The available error estimates allow us to conclude a discrete in
time $L^{\infty}_t(L^2_{\ssy P}(L^2_x))$ convergence of the
Backward Euler fully-discrete approximations of ${\widehat u}$.
%
%
%
%
\begin{theorem}\label{FFQEWR}
Let $r=2$ or $3$, $\nu(r)$ be defined by \eqref{Nice_Day_1},
${\widehat u}$ be the solution of problem \eqref{AC2}, and
$({\widehat U}_h^m)_{m=0}^{\ssy M}$ be the Backward Euler finite
element approximations of ${\widehat u}$ constructed by
\eqref{FullDE1}-\eqref{FullDE2}.
Then, there exists a nonnegative constant $C$, independent of $h$,
$\Delta\tau$, $\Delta{t}$ and $\Delta{x}$, such that:
if $\mu^2\,\Delta\tau\leq\frac{1}{4}$, then
\begin{equation*}
\max_{0\leq{m}\leq{\ssy M}}\left\{{\mathbb E}\left[ \|{\widehat
U}_h^m-{\widehat u}(\tau_m,\cdot)\|_{\ssy 0,D}^2\right]
\right\}^{\half} \leq\,C \,\Big[\,\omega_*(\Delta\tau,\epsilon_1)
\,\,\Delta\tau^{\frac{1}{8}-\epsilon_1}
+\epsilon_2^{-\frac{1}{2}}\,\,\,h^{\nu(r)-\epsilon_2}\,\Big]
\end{equation*}
forall $\epsilon_1\in(0,\tfrac{1}{8}]$ and $\epsilon_2\in(0,\nu(r)]$,
where $\omega_*(\Delta\tau,\epsilon_1):=\epsilon_1^{-\frac{1}{2}}
+(\Delta\tau)^{\epsilon_1}(1+(\Delta{\tau})^{\frac{7}{4}}
+(\Delta{\tau})^{\frac{3}{4}})^{\half}$.
\end{theorem}
%
%
%
%
%
%
%
%
%
%
%
\begin{proof}
The estimate is a simple consequence of the error bounds
\eqref{Lasso1} and \eqref{ElPasso}.
\end{proof}
%
%
%
\section*{Acknowledgments}
Work partially supported by the European Union's Seventh
Framework Programme
(FP7-REGPOT-2009-1) under grant agreement no. 245749
‘Archimedes Center for Modeling, Analysis and Computation’
(University of Crete, Greece).
%
%
%
%
\def\cprime{$'$}
\def\cprime{$'$}
%

%
%
%
\appendix
\section{ }\label{APP_1}
\par\noindent
Let $t>0$ and $\mu_k:=\lambda_k^2(\lambda_k^2-\mu)$ for
$k\in{\mathbb N}$. First, we recal that
${\mathcal S}(t)w_0=\sum_{k=1}^{\infty}e^{-\mu_k\,t}\,(w_0,\varepsilon_k)_{\ssy 0,D}\,\varepsilon_k$
for $t\ge 0$, and set ${\widetilde{\mathcal S}}(t)w_0=e^{-\mu^2\,t}\,{\mathcal S}(t)w_0$ for $t\ge 0$.
Next, follow Chapter~3 in \cite{Thomee}, to obtain
\begin{equation*}
\begin{split}
\big\|\partial_t^{\ell}{\widetilde{\mathcal S}}(t)w_0\big\|^2_{\ssy{\bfdot H}^p}
&=\sum_{k=1}^{\infty}\lambda_k^{2p}\,
\big(\partial_t^{\ell}{\widetilde{\mathcal S}}(t)w_0,
\varepsilon_k\big)^2_{\ssy 0,D}\\
&=\sum_{k=1}^{\infty}\lambda_k^{2p}\,(\mu_k+\mu^2)^{2\ell}\,
\big({\widetilde{\mathcal S}}(t)w_0,\varepsilon_k\big)^2_{\ssy 0,D}\\
&=\sum_{k=1}^{\infty}\lambda_k^{2p}\,(\mu_k+\mu^2)^{2\ell}\,
e^{-2\,(\mu_k+\mu^2)\,t}
\,\big(w_0,\varepsilon_k\big)^2_{\ssy 0,D},\\
\end{split}
\end{equation*}
which yields
\begin{equation}\label{Reggo5}
\big\|\partial_t^{\ell}{\widetilde{\mathcal S}}(t)w_0\big\|^2_{\ssy{\bfdot
H}^p}\leq\,{\widetilde C}_{\mu,\ell}\,\sum_{k=1}^{\infty}
\,\lambda_k^{2(p+4\ell)}\,
e^{-\lambda^4_k\,t}\,(w_0,\varepsilon_k)^2_{\ssy 0,D},
\end{equation}
where ${\widetilde C}_{\mu,\ell}:=\left(1+\frac{\mu}{\pi^2}
+\frac{\mu^2}{\pi^4}\right)^{2\ell}$.
Now, use \eqref{Reggo5}, to have
\begin{equation*}
\begin{split}
\int_{t_a}^{t_b}(\tau-t_a)^{\beta}\,\big\|\partial_t^{\ell}
{\widetilde{\mathcal S}}(\tau)w_0\big\|^2_{\ssy {\bfdot H}^p}\,d\tau
\leq&\,{\widetilde C}_{\mu,\ell}\,\sum_{k=1}^{\infty}\lambda_k^{2(p+4\ell-2\beta)}\,
\Big(\int_{t_a}^{t_b}[\lambda_k^4(\tau-t_a)\big]^{\beta}
\,e^{-\lambda_k^4\,\tau}\,d\tau\Big)
\,(w_0,\varepsilon_k)^2_{\ssy 0,D}\\
\leq&\,{\widetilde C}_{\mu,\ell}\,\sum_{k=1}^{\infty}\lambda_k^{2(p+4\ell-2\beta-2)}\,
\Big(\int_0^{\lambda_k^4\,(t_b-t_a)}
\rho^{\beta}\,e^{-(\rho+\lambda_k^4t_a)}\,d\rho\Big)
\,(w_0,\varepsilon_k)^2_{\ssy 0,D}\\
\leq&\,{\widetilde C}_{\mu,\ell}\,\Big(\int_0^{\infty}\rho^{\beta}
\,e^{-\rho}\,d\rho\Big)\,
\sum_{k=1}^{\infty}\lambda_k^{2(p+4\ell-2\beta-2)}\,
\,(w_0,\varepsilon_k)^2_{\ssy 0,D},\\
\end{split}
\end{equation*}
which yields
\begin{equation}\label{Reggo6}
\int_{t_a}^{t_b}(\tau-t_a)^{\beta}\,\big\|\partial_t^{\ell}
{\widetilde{\mathcal S}}(\tau)w_0\big\|^2_{\ssy {\bfdot H}^p}\,d\tau
\leq\,{\widetilde C}_{\beta,\ell,\mu}\,
\, \|w_0\|^2_{\ssy {\bfdot
H}^{p+4\ell-2\beta-2}},
\end{equation}
where ${\widetilde C}_{\beta,\ell,\mu}={\widetilde C}_{\mu,\ell}
\,\int_0^{\infty}x^{\beta}\,e^{-x}\,dx$. Observing that
$\partial_t^{\ell}{\mathcal S}(t)w_0=e^{\mu^2\,t}
\,\sum_{m=0}^{\ell}\binom{\ell}{m}\,\,\mu^{2(\ell-m)}
\,\,\partial_t^{m}{\widetilde{\mathcal S}}(t)w_0$,
and using \eqref{Reggo6}, we conclude that
\begin{equation*}
\int_{t_a}^{t_b}(\tau-t_a)^{\beta}\,\big\|\partial_t^{\ell}
{\mathcal S}(\tau)w_0\big\|^2_{\ssy {\bfdot H}^p}\,d\tau
\leq\,e^{2\,\mu^2\,T}\,\,C_{\beta,\ell,\mu}\,\,
\sum_{m=0}^{\ell}\|w_0\|^2_{\ssy{\bfdot H}^{p+4m-2\beta-2}}
\end{equation*}
which yields \eqref{Reggo3} with ${\mathcal C}_{\beta,\ell,\mu,\mu T}
=C_{\beta,\ell,\mu}\,e^{2\,\mu^2\,T}\,\ell$. $\mybox$
%
%

\begin{thebibliography}{10}
%
%
%
\bibitem{ANZ}
E.J.~Allen, S.J.~Novosel and Z.~Zhang.
\newblock Finite element and difference approximation of some linear
stochastic partial differential equations.
\newblock {\em Stochastics Stochastics Rep.},
vol. 64, pp. 117--142, 1998.
%
%
%
\bibitem{AKS}
A.~Are, M.A.~Katsoulakis and A.~Szepessy.
\newblock Coarse-Grained Langevin Approximations and
Spatiotemporal Acceleration for Kinetic
Monte Carlo Simulations of Diffusion
of Interacting Particles.
\newblock {\em Chin. Ann. Math.},
vol. 30B(6), pp.  653--682, 2009.
%
%
%
%
%
%
%
\bibitem{BinLi}
L.~Bin.
\newblock{Numerical method for a parabolic stochastic partial differential equation.}
\newblock{Master Thesis 2004-03,}
\newblock{Chalmers University of Technology, G{\"o}teborg, Sweden, June 2004.}
%
%
%
%
%
%
%
\bibitem{BlomMPW}
D.~Bl{\"o}mker, S.~Maier-Paape and T.~Wanner.
\newblock{Second phase spinonal decomposition for the
Cahn-Hilliard-Cook equation.}
\newblock{\em Transactions of the AMS}, 360 (2008), pp.~449-489.
%
%
%
%
%
%
\bibitem{BrHilbert1970}
J.H. Bramble and S.R. Hilbert.
\newblock{Estimation of linear functionals on Sobolev spaces
with application to Fourier transforms and spline interpolation}.
\newblock{\em SIAM J. Numer. Anal.}, 7 (1970), pp.~112-124.
%
\bibitem{DZ2007}
A.~Debussche and L.~Zambotti.
\newblock Conservative Stochastic Cahn-Hilliard equation with reflection.
\newblock {\em Annals of Probability},
vol. 35, pp. 1706-1739, 2007.
%
%
%
%
\bibitem{DunSch}
N. Dunford and J.T.~Schwartz.
\newblock{\em Linear Operators. Part II. Spectral Theory.
Self Adjoint Operators in Hilbert Space.}
\newblock{\rm Reprint of the 1963 original. Wiley Classics
Library. A Wiley-Interscience Publication. John Wiley \& Sons,
Inc.. New York, 1988.}
%
%
%
%
%
%
%
%
%
%
%
%
%
%
%
%
%
%
%
%
%
%
%
\bibitem{GK}
W.~Grecksch and P.E.~Kloeden.
\newblock{Time-discretised Galerkin approximations of parabolic stochastic
PDEs}.
\newblock{\em Bull. Austral. Math. Soc.},
vol. 54, pp. 79--85, 1996.
%
%
%
%
%
%
%
%
%
%
%
%
%
%
%
%
%
%
%
%
%
%
%
%
%
\bibitem{Golub}
G.~H. Golub and C.~F. Van Loan.
\newblock {\em Matrix Computations}.
\newblock {Second Edition,
The John Hopkins University Press, Baltimore, 1989}.
%
%
%
%
%
%
%
%
%
%
\bibitem{Hohen}
P. C. Hohenberg and B.I. Halperin.
\newblock {Theory of dynamic critical
phenomena}.
\newblock  {\em J. Rev. Mod. Phys.}
vol. 49, pp. 435--479, 1977.
%
%
%
%
%
%
%
%
%
%
%
%
%
\bibitem{KXiong}
G.~Kallianpur and J.~Xiong.
\newblock{\em Stochastic Differential Equations in
Infinite Dimensional Spaces}.
\newblock{Institute of Mathematical Statistics,
Lecture Notes-Monograph Series vol. 26,
Hayward, California, 1995.}
%
%
\bibitem{KV03}
M.A~Katsoulakis and D.G.~Vlachos.
\newblock{ Coarse-grained stochastic processes and kinetic Monte Carlo
simulators for the diffusion of interacting particles.}
\newblock{\em J. Chem. Phys.,}
vol.  119, pp. 9412--9427, 2003.
%
%
%
%
%
\bibitem{KloedenShot}
P.E.~Kloeden and S.~Shot.
\newblock Linear-implicit strong schemes for It{\^o}-Galerkin
approximations of stochastic PDEs.
\newblock{\em Journal of Applied Mathematics and Stochastic Analysis.},
vol. 14, pp. 47--53, 2001.
%
%
\bibitem{KZ2010}
G.T.~Kossioris and G.E.~Zouraris,
\newblock{Fully-discrete
finite element approximations for a fourth-order linear stochastic
parabolic equation  with additive space-time white noise},
\newblock{Mathematical Modelling and Numerical Analysis
{\bf 44}, 289-322 (2010)}.
%
%
\bibitem{KZ2009}
G.T.~Kossioris and G.E.~Zouraris,
\newblock{Finite element approximations for a linear fourth-order parabolic SPDE
in two and three space dimensions with additive space-time white noise},
\newblock{http://dx.doi.org/doi:10.1016/j.apnum.2012.01.003,
Applied Numerical Mathematics (to appear)}.
%
%
%
%
\bibitem{LM2009}
S.~Larsson and A.~Mesforush,
\newblock{Finite element approximation
of the linearized Cahn-Hilliard-Cook equation},
\newblock{IMA J. Numer. Anal.
{\bf 31}, 1315-1333 (2011)}.
%
%
%
\bibitem{LMag}
J.L.~Lions and E.~Magenes.
\newblock {\em Non-Homogeneous Boundary Value Problems and Applications. Vol. I}.
\newblock { Springer--Verlag, Berlin - Heidelberg, 1972.}
%
%
%
%
%
%
%
%
%
%
\bibitem{Printems}
J.~Printems.
\newblock
On the discretization in time of parabolic stochastic partial
differential equations.
\newblock {\em Mathematical Modelling and Numerical Analysis},
vol. 35, pp. 1055--1078, 2001.
%
%
%
\bibitem{Rogers1988}
T.M.~Rogers, K.R.~Elder and R.C.~Desai.
%
\newblock
Numerical study of the late stages of spinodal decomposition.
%
\newblock {\em Physical Review B},
vol. 37, pp. 9638--9651, 1988.
%
%
%
%
\bibitem{Schatz1974}
%
A.~H.~Schatz.
\newblock
An observation concerning Ritz-Galerkin methods with indefinite
bilinear form.
\newblock{\em Math. Comp.}, vol. 28, pp. 959-962, 1974.
%
%
%
%
%
\bibitem{Thomee}
V.~Thom{\'e}e.
\newblock{\em Galerkin Finite Element Methods for Parabolic Problems},
\newblock{\em Spriger Series in Computational Mathematics} vol. 25,
\newblock Springer-Verlag, Berlin Heidelberg, 1997.
%
%
%
%
%
%
%
%
%
%
%
\bibitem{YubinY05}
Y.~Yan.
\newblock Galerkin Finite Element Methods for Stochastic
Parabolic Partial Differential Equations.
\newblock {\em SIAM Journal on Numerical Analysis},
vol. 43, pp. 1363--1384, 2005.
%
%
%
%
%
%
%
%
\bibitem{Walsh86}
J.B.~Walsh.
\newblock{An introduction to stochastic partial differential
equations.}
\newblock
{Lecture Notes in Mathematics no. 1180,
pp. 265--439, Springer Verlag,
Berlin Heidelberg, 1986.}
%
%
%
%
%
\bibitem{Walsh05}
J.B.~Walsh.
\newblock {Finite Element Methods for Parabolic Stochastic PDEs.}
\newblock {\em Potential Analysis},
vol. 23, pp. 1--43, 2005.
%
\end{thebibliography}
\end{document}